\documentclass{article}
\usepackage{amsmath,amsfonts,amssymb,amsthm,url,color,epsfig,graphics,dsfont}
\usepackage[colorlinks=true]{hyperref}

\def\R{\mathrm{I\kern-0.21emR}}
\def\N{\mathrm{I\kern-0.21emN}}
\newcommand{\Z}{{\mathbb{Z}}}

\newcommand{\Tr}{\mathrm{Tr}}

\renewcommand{\geq}{\geqslant}
\renewcommand{\leq}{\leqslant}

\usepackage{xcolor}

\newtheorem{theorem}{Theorem}  
\newtheorem{proposition}{Proposition}
\newtheorem{corollary}{Corollary}

\newtheorem{lemma}{Lemma}

\theoremstyle{definition}\newtheorem{remark}{Remark}

\title{Weyl formulae for some singular metrics with application to acoustic modes in gas giants}
\author{Yves Colin de Verdi\`ere\footnote{Universit\'e Grenoble-Alpes, Institut Fourier, Unit{\'e} mixte de recherche CNRS-UGA 5582, BP 74, 38402-Saint Martin d'H\`eres Cedex (France) (\texttt{yves.colin-de-verdiere@univ-grenoble-alpes.fr})},
\and
Charlotte Dietze\footnote{Sorbonne Universit\'e, CNRS, Universit\'e Paris Cit\'e, Laboratoire Jacques-Louis Lions (LJLL), F-75005 Paris, France (\texttt{charlotte.dietze@sorbonne-universite.fr})}, 
\and
Maarten V. de Hoop\footnote{Simons Chair in Computational and Applied Mathematics and Earth Science, Rice University (\texttt{mvd2@rice.edu})},
 \and
Emmanuel Tr\'elat\footnote{{\bf Corresponding author.} Sorbonne Universit\'e, CNRS, Universit\'e Paris Cit\'e, Inria, Laboratoire Jacques-Louis Lions (LJLL), F-75005 Paris, France (\texttt{emmanuel.trelat@sorbonne-universite.fr})}
}

\date{}

\begin{document}

\maketitle

\begin{abstract}
This paper is motivated by recent works on inverse problems for acoustic wave propagation in the interior of gas giant planets. In such planets, the speed of sound is isotropic and tends to zero at the surface. Geometrically, this corresponds to a Riemannian manifold with boundary whose metric blows up near the boundary. Here, the spectral analysis of the corresponding Laplace-Beltrami operator is presented and the Weyl law is derived. The involved exponents depend on the Hausdorff dimension which, in the supercritical case, is larger than the topological dimension.
\end{abstract}

\medskip

\noindent\emph{Note on this arXiv version.} This version differs from the published one (Ann. Henri Poincar\'e, 2026, \url{https://doi.org/10.1007/s00023-026-01678-z}) by the proof of Proposition~\ref{prop_parametrix}, in which a gap, kindly pointed out to us by Luc Hillairet, has been fixed (see Remark~\ref{rem_gap}). The gap concerned the proof only: the statement of Proposition~\ref{prop_parametrix}, as well as all other results of the paper, are unaffected.

\smallskip


\medskip


\section{Introduction}\label{sec1}

\subsection{Seismology on gas giant planets}

Seismology has played an important role in revealing the (deep) interiors of gas giant planets in our solar system (see \cite{christensen-dalsgaardHelioseismology2002, markhamGasPlanetSeismology2021}). Indeed, the acoustic spectra and free oscillations have been studied for Saturn and Jupiter over the past few decades  (see \cite{gaulmeSeismologyGiantPlanets2015, gudkovaFreeOscillationsJupiter1999, vorontsovFreeOscillationsJupiter1976}). The excitation of acoustic modes in gas giant planets presumably occurs through convection in their interiors. The observation of acoustic eigenfrequencies, that is, the discrete spectrum can be realized, in principle, through visible photometry, thermal infrared photometry, Doppler spectrometry, and ring seismology for nonradial oscillations (see \cite{marley1990nonradial, marley1991nonradial}), in particular, in the case of Saturn. In ring seismology and with the Cassini mission, one measured the ``resonances'' in the inner C ring of Saturn with visual and infrared mapping spectrometer (VIMS) stellar occultations (see \cite{dewberryConstrainingSaturnInterior2021, Frenchetal2021, Mankovich2020}). The rings are gravitationally coupled to the acoustic modes of the planet (taking self-gravitation into account). Detection of Jupiter's acoustic eigenvalues has been attempted with ground-based imaging-spectrometry (seismographic imaging interferometer for monitoring of planetary atmospheres or SYMPA) by measuring line of sight velocity (see \cite{gaulmeSeismologyGiantPlanets2015, Schmideretal2007}). Recently, Juno spacecraft gravity measurements have provided evidence for normal modes of Jupiter (see \cite{Duranteetal2022}).

Ring seismology in Saturn (via density waves in the C~ring) and Doppler/gravity measurements in Jupiter provide direct handles on global acoustic modes. These datasets are sensitive not only to the discrete spectrum (mode frequencies) but also to interior physics such as differential rotation and stable stratification, which modulate the coupling of planetary normal modes with the rings and atmospheric signals. This motivates a spectral analysis in which the measurable eigenfrequencies arise from an acoustic operator tied to the near-surface sound-speed profile.

\subsection{Singular Riemannian metrics}\label{sec_singular_metric}

On a gas giant planet, unlike a rocky planet, the speed of sound goes to zero at the boundary. In the geometric mathematical model that we employ hereafter, the rate at which this happens follows a power law which determines a specific conformal blow-up rate of a Riemannian metric, thus defining a singular metric. This rate happens to be slower than on asymptotically hyperbolic manifolds and the boundary is at a finite distance from interior points. The rate is implied by an equation of state in the upper part of the planet, in general, in the sense of a fit. (For some models of the speed of sound of Jupiter and Saturn showing this behavior, see \cite[Figure~1]{gudkovaFreeOscillationsJupiter1999} and \cite[Figure~1]{Mankovichetal2019}.) Only for a polytrope is the rate exact. Polytropes, for which the pressure is proportional to a power of the density of mass, have been viewed as relevant simplifications; models with variable polytrope index have indeed been applied to planet and material models (see \cite{Weppneretal2015}). Typically, an equation of state is computed numerically using density functional molecular dynamics simulations with mixtures of chemical elements: The dominant elements in terms of mass fraction are hydrogen and helium, but also heavy elements are important. The equation of state is different for the upper part and the deep interior as the helium fraction can be higher in the interior due to helium rain (helium becoming immiscible with hydrogen at high pressure). Equations of state play a vital role in the evolution and realization of structure of gas giant planets (see \cite{militzerUnderstandingJupiterInterior2016, militzerModelsSaturnInterior2019}).

\medskip

More specifically, if $g_e$ is the Euclidean Riemannian metric on a smooth domain $X \subset \R^{n+1}$, then the speed of sound $c$ can be encoded by the conformally Euclidean Riemannian metric $g = c^{-2} g_e$. In local coordinates where the boundary of $X$ is (locally) described by $u=0$, the polytropic model suggests that $c \sim u^{1/2}$. Indeed, the natural generalization is $c \sim u^{\alpha/2}$, that is, $c^{-2} \sim u^{-\alpha}$; through previous analysis done in \cite{MdH-1} it appears that restricting $\alpha$ according to $\alpha \in (0,2)$ guarantees the presence of a discrete spectrum as it has been observed. Thus, the Riemannian geometry lies between standard geometry with boundary and asymptotically hyperbolic geometry. Some of the phenomena in this geometry are unlike those seen at either end. The extreme case $\alpha = 0$ corresponds physically to solid bodies and mathematically to manifolds with boundary, and the other extreme $\alpha = 2$ corresponds to asymptotically hyperbolic geometry but is far from all planetary models.

Therefore, following \cite[Section~1.1]{MdH-1}, we model a gas giant planet as a smooth manifold $X$ with a boundary, endowed with a Riemannian metric $g$ on $X \setminus \partial X$ such that, near $\partial X$, we have $g = \bar g / u^\alpha$ where $\bar g$ is a well-defined Riemannian metric up to the boundary, and $\partial X = \{u=0\}$ locally. The fact that $\bar g$ is neither zero nor infinite at $\partial X$ implies a specific blow-up rate for $g$ near $\partial X$. This conformal power-law blow-up is the key geometric feature of gas giant metrics. The speed of sound might contain jump discontinuities where phase transitions occur (see \cite{militzerStudyJupiterInterior2024}), that is, the metric can contain conormal singularities while the manifold consists of multiple ``layers'', which our framework accommodates at the level of heat asymptotics and spectral localization.
A key interior boundary in gas giants corresponds with the transition from molecular to metallic hydrogen. Accounting for discontinuities in an asymptotic formalism for gas giant seismology was developed a few decades ago (see \cite{Provostetal1993}). 

\medskip

Starting from the incremental Lagrangian formulation, the acoustic limit identifies the wave speed by $c^2=\kappa/\rho_0$, with $\kappa=P_0\gamma$ the adiabatic bulk modulus. Under the standard Cowling truncation, the pressure satisfies a scalar acoustic equation whose principal symbol is $\tau^2 - c^2\,|\xi|^2$.
Thus, pressure singularities propagate along the geodesics of the \emph{acoustic metric} $g=c^{-2}\,g_e$. This scalar reduction underpins our use of the Laplace-Beltrami operator associated with $g$ to model high-frequency mode counts and their spatial concentration. 
We refer to \cite[Sections~5.3--5.5]{MdH-1} for details on the incremental Lagrangian formulation, the Cowling approximation, and the identification of $c^2=\kappa/\rho_0$ with ray propagation along $g=c^{-2}g_e$.

In the present spectral study we consider the isotropic, non-rotating acoustic limit (setting $\Omega=0$ in the reference model and neglecting buoyancy terms), which isolates the role of the near-surface $c$-profile (see \cite[Section~5.3]{MdH-1}).

\medskip

At a free surface, the natural linearized boundary condition is the vanishing of the dynamic pressure. In the gas-giant setting where $c\to0$ toward the boundary, one can impose this condition on a slightly \emph{truncated} domain located just below the surface. 
For a consistent free-surface condition and spectral convergence under near-boundary truncation, see \cite[Section~5.4]{MdH-1}.
The corresponding spectral problem is well-posed: the associated operator has discrete spectrum, and the eigenvalues on the truncated domain converge at a controlled rate to those of the limiting (degenerate) operator on the full manifold. This provides a consistent link between physical boundary modeling and the spectral framework used here.

\medskip

The mathematical study of the spectrum associated with gas giants' acoustic modes was initiated in \cite{MdH-1}. In this paper, we analyze the relevant Laplace-Beltrami operator and we compute the Weyl law. The study of Weyl asymptotics, which reflects some properties of the singular metric, is a preliminary step towards analyzing some inverse problems, in view of reconstructing some features of the internal structure of gas giant planets.

For a thorough discussion on how the singular metrics of the type considered here arise from equations of state on gas giants, we refer to \cite[Section~1.2]{MdH-1}. The basic, polytropic equation of state leads precisely to a blow up at a rate $\alpha=1$, above the critical value introduced in the main text. Observations show the presence of a discrete acoustic spectrum; this can only be ensured if $\alpha < 2$.

\section{Mathematical model and main results}

\subsection{Mathematical model}\label{sec_math_model}

Let $X$ be a smooth compact manifold of dimension $n+1$ with a boundary $\partial X$. Near $\partial X$, $X$ is diffeomorphic to $[0,1) \times M$, where $M$ is a smooth compact manifold of dimension $n \geq 1$ and $\partial X$ is identified with $\{0\} \times M$ and also with $u=0$ where $u$ is a transverse coordinate, locally near $\partial X$, ranging over $[0,1)$. As discussed in Section \ref{sec_singular_metric}, we consider on $X$ a singular Riemannian metric $g$ that is a smooth metric on $X\setminus\partial X$, written near $\partial X$ as 
\begin{equation}\label{initial_metric}
g = \bar{g}/u^\alpha  
\end{equation}
where $0 < \alpha  < 2$ and $\bar{g}$ is a smooth (non-singular) Riemannian metric on $X$, up to the boundary. Following \cite[Proposition 2]{MdH-1}, which uses a normal form for the metric near the boundary, due to \cite[Lemma 5.2]{Graham}, we have
$$
g = u^{-\alpha} (du^2 + g_0(u)) 
$$
where $g_0(u)$ is a Riemannian metric on $M$ (pulled back to the level set $u = \mathrm{Cst}$) depending smoothly on $u \in [0,1)$. 

We make a change of variable. Setting $x = x(u) = \int_0^u s^{-\alpha/2} ds = (1-\frac{\alpha}{2})^{-1} u^{1-\frac{\alpha}{2}}$, we get
\begin{equation}\label{gg1}
g = dx^2 + x^{-\beta} g_1(x)\ \text{where}\ \beta = \frac{2 \alpha}{2 - \alpha}
\end{equation}
and $g_1(x)$ is a Riemannian metric on $M$ (pulled back to the level set $x=\mathrm{Cst}$), depending smoothly on $x \in (0,1)$ and that is continuous at $x=0$. We note that, since $\alpha \in (0,2)$, $\beta$ can take any positive value. We also note that a polytrope (for any index) corresponds to $\beta = \beta_{\textrm{poly}} = 2$. We have that $g_1(x) = C(\alpha) g_0(u)$ for some constant $C(\alpha) > 0$. 

For any $x\in[0,1)$, denoting by $dv_1^x$ the volume measure on $M$ associated to the metric $g_1(x)$, the $g$-volume is $dv_g = x^{-\beta n/2 } \, \vert dx\vert \, dv_1^x$. 
The volume is finite if and only if $\beta < \beta_c$, where
\begin{equation}\label{def_betac}
\beta_c = \frac{2}{n}
\end{equation}
is a critical value of $\beta$. We will see later that this critical value plays a role in the Weyl asymptotics. At this point, we can note that $\beta_{\textrm{poly}} > \beta_c$ for $n = 2$. 

The following three propositions were proved in \cite{MdH-1}. The first proposition concerns the Hausdorff dimension.

\begin{proposition}\label{lem_Hausdorff}
The Hausdorff dimension of $(X,g)$ is 
$$
d_H = \max\left(n+1, n \left(1+\frac{\beta}{2}\right)\right).
$$
\end{proposition}

We define $\delta_H = n (1 + \frac{\beta}{2})$, and note that $d_H > n+1$ ($n+1$ is the topological dimension of $X$) if and only if $\beta > \beta_c$. We give in Appendix \ref{app_hormander} a sketch of the proof of Proposition~\ref{lem_Hausdorff}, in which we also show that $d_H$ coincides with the Minkowski dimension of $(X,g)$.

\begin{proposition}\label{lem_EAS}
The Laplace-Beltrami operator $\triangle_g$, with core $C_0^\infty(X \setminus \partial X)$, is essentially self-adjoint if and only if $\beta \geq \beta_c$. 
\end{proposition}

For $\beta < \beta_c$, there exist several extensions of $\triangle_g$, with core $C_0^\infty (X \setminus \partial X)$. In the further analysis, we consider its Friedrichs extension (that is, ``Dirichlet extension''). As we consider ``pressure'' as the physical quantity satisfying the partial differential equation, the natural boundary condition is the Dirichlet one --- this corresponds with a pressure-free boundary condition, like for oceans.

In the subcritical case, when $\triangle_g$ is not essentially self-adjoint, selecting a self-adjoint extension amounts to prescribing a boundary behavior at the free surface $x=0$. The Friedrichs extension is physically canonical for three reasons. First, it is the energy (Dirichlet-form) extension: its domain is the closure of $C_c^\infty$ under the quadratic form $u\mapsto\int_X |\nabla_g u|^2\,dv_g$, hence it selects finite-energy states and excludes boundary-singular modes. Second, it arises as the limit of physically truncated models: imposing the free-surface (pressure-release) condition on $\{x=\varepsilon\}$ and letting $\varepsilon\rightarrow 0^+$ yields, in norm resolvent, the Friedrichs extension on $\{x>0\}$. Third, for the normal-form operator $-\partial_x^2+\frac{C}{x^2}$ (see further) with $C<3/4$, all self-adjoint extensions correspond to mixing the two admissible asymptotics $x^{1/2\pm\nu}$; the Friedrichs extension selects the more decaying branch, which corresponds to vanishing acoustic energy flux into the boundary.
Finally, in the subcritical regime the leading Weyl term is bulk-dominated and independent of the extension; only lower-order boundary corrections may change. Thus the Friedrichs choice both matches the physical free surface and leaves the principal asymptotics unaffected.

\begin{proposition}\label{lem_spectrum_discrete}
For every $\beta > 0$, the spectrum of $\triangle_g$ is discrete. 
\end{proposition}

We denote the eigenvalues of $\triangle_g$ by $0 < \lambda_1 \leq \lambda_2 \leq \cdots \leq \lambda_j \leq \cdots$ with associated eigenfunctions $\phi_j$, $j \in\N^*$, normalized in $L^2(X,dv_g)$. We note that, if the volume of $X$ is infinite, i.e., if $\beta \geq \beta_c$, we have $\lambda_1 > 0$, in contrast to the usual Riemannian case. The Weyl counting function is defined by
$$ 
N(\lambda) = \# \{j\in \N^*\ \mid\  \lambda_j \leq \lambda \} 
$$
where the eigenvalues are counted with their multiplicity. Our objective is to derive a Weyl law describing the asymptotics of $N(\lambda)$ as $\lambda\rightarrow+\infty$. 

\begin{remark}
The following fact will be used in Section~\ref{sec:gene}. For any $\varepsilon>0$, there exists $\delta>0$ such that the metric $g$ is \emph{$\varepsilon$-quasi-isometric} (see Appendix~\ref{app:quasi}) to a singular Riemannian metric $\tilde{g}$ on $X$, smooth on $X \setminus \partial X$ and given by $\tilde{g} = dx^2+x^{-\beta}g_1(0)$ on $(0,\delta) \times M$. In order to derive a Weyl law for $(X,g)$ it suffices to derive the corresponding Weyl law for $\tilde{g}$ for any $\varepsilon>0$ (see, again, Appendix~\ref{app:quasi} for details). This remark is important, because it implies that we mainly have to work within the so-called \textit{separable} case.
\end{remark}

\paragraph{Separable case.}
We say that we are in the \textit{separable} case if the metric $g_1(x)$ on $M$ (defined by \eqref{gg1}) does not depend on $x$, i.e., $g_1(x)=g_1(0)$ for any $x \in(0,1)$; we still denote this metric by $g_1$. In the sequel, we consider $[0,1) \times M$ instead of $[0,\delta) \times M$ for simplicity of notation, while the proofs are similar in both cases. 

We denote by $\triangle_M$ the Laplace-Beltrami operator on $(M,g_1)$. We denote the eigenvalues of $\triangle_M$ by $0 \leq \omega_1\leq\omega_2\leq\cdots\leq\omega_j\leq\cdots$ with an associated orthonormal basis of eigenfunctions $(\psi_j)_{j\in\N^*}$. The Weyl counting function for $\triangle_M$ is defined by 
$$ 
N_M(\omega) = \# \{ k \in \N^*\ \mid\ \omega_k\leq \omega \} .
$$
Since $g_1$ is a smooth Riemannian metric on $M$, the classical Weyl law for $(M,g_1)$ yields that
$N_M(\omega)= \gamma_n \mathrm{Vol}_{g_1}(M) \omega^{n/2}+\mathrm{O}\left( \omega ^{n/2} \right) $ as $\omega \rightarrow +\infty$ where 
\begin{equation}\label{def_gamman}
\gamma_n = \frac{1}{(4\pi)^{n/2} \Gamma\left(\frac{n}{2} + 1\right)}
\end{equation}
(see \cite[Chapter 3E]{BGM} for the heat trace and then apply the Karamata tauberian theorem, i.e., Theorem~\ref{thm_karamata} in Appendix~\ref{app:ka}).

\medskip

Denoting by $dv_1$ the volume measure on $M$ associated to the metric $g_1=g_1(0)$, the $g$-volume is $dv_g = x^{-\beta n/2 }\, \vert dx\vert \, dv_1$. 
Making the change of function $f\mapsto  x^{-\beta n/4}f$, we get the new volume form $\vert dx\vert \, dv_1$; the Laplace-Beltrami operator on $X_1=(0,1) \times M$ is now given by
$$
\triangle_g =-\partial_x^2+ \frac{C_\beta}{x^2} +x^\beta \triangle_M 
$$
where $x\in(0,1)$ and
$$
C_\beta = \frac{\beta n}{4}\left( \frac{\beta n}{4}+1 \right) .
$$
The proof is straightforward by performing an integration by parts with respect to $x$ in the Dirichlet form defining the Laplace-Beltrami operator, using the Dirichlet boundary condition at $x=0$. We note that $C_\beta \geq 3/4$ if and only if the volume of $X$ is infinite. 
Using the Weyl criterion (see Appendix~\ref{app:weyl}), if $C_\beta > 3/4$ (i.e., if $\beta>\beta_c$) then $P_\omega$ defined by \eqref{equ:p1} below is essentially self-adjoint for any $\omega \geq 0$; if $C_\beta<3/4$ we take the Friedrichs extension (including for $\omega>0$).

We will need to work first on the non-compact conic manifold $X_\infty =(0,+\infty) \times M$ endowed with the metric $g=dx^2+x^{-\beta}g_1$. Let $\triangle_\infty$ stand for the Laplace-Beltrami operator on $(X_\infty,g)$. Invoking a separation of variables, we have
\begin{equation} \label{equ:direct}
\triangle_\infty = \underset{k=1}{\overset{+\infty}{\oplus}} \left(\mathrm{id}\otimes \pi_k \right) \left( P_{\omega_k}\otimes \mathrm{id}\right)\left( \mathrm{id}\otimes \pi_k \right) 
\end{equation}
where 
\begin{equation} \label{equ:p1}  
P_\omega = -\partial_x^2 + \frac{C_\beta}{x^2} +\omega x^\beta 
\end{equation}
is a Schr\"odinger operator on $L^2((0,+\infty), dx)$ for any $\omega\geq 0$, and where $\pi_k$ is the orthogonal projection of $L^2(M,dv_1)$ onto the subspace generated by $\psi_k$ and $\mathrm{id}$ denotes the identity operator on $L^2(X_\infty, \vert dx\vert\, dv_1) $ (resp., on $L^2(M, dv_1)$). Hence, $\triangle_\infty$ is unitarily equivalent to $\underset{k=1}{\overset{+\infty}{\oplus}} P_{\omega_k}$.

\subsection{Main results}

Recalling that $g_1$ is defined by \eqref{gg1}, we set $G = g_1(0)$ and denote by $v_G$ the corresponding volume form on $M$. We also recall that $\beta_c$ is defined in \eqref{def_betac} and that $\gamma_n$ is defined in \eqref{def_gamman}.

\subsubsection{Weyl asymptotics}

\begin{theorem}\label{theo:A} 
\begin{itemize}
\item 
If $\beta > \beta_c$ then
$$ 
N(\lambda) \sim A(\beta, n) v_G(M)\,\lambda ^{d_H/2} 
$$
as $\lambda \rightarrow +\infty$, with 
$$ 
A(\beta,n)= \frac{n \gamma_n (\beta +2)}{4\Gamma (1+d_H/2)}\int_0^{+\infty} \mathcal{Z}_1 (\tau ) \tau^{\frac{d_H}{2}-1} d\tau 
$$
where $\mathcal{Z}_1(\tau) = \Tr(\mathrm{exp}(-\tau P_1))$ and $P_1$ is the Schr\"odinger operator on $L^2 ((0,+\infty),dx)$ defined by \eqref{equ:p1}.
\item If $\beta=\beta_c=2/n $ then
$$ 
N(\lambda)\sim C_n v_G(M)\,\lambda^{(n+1)/2} \ln \lambda 
$$
as $\lambda \rightarrow+\infty $, with
$$ 
C_n= \frac{1}{(n+1) (4\pi)^{(n+1)/2} \Gamma\left( (n+1)/2 \right)}. 
$$
In particular, $C_1 = 1/8\pi$. 
\item If $\beta < \beta_c $ then
$$ 
N(\lambda)\sim \gamma_{n+1} v_g(X)\,\lambda ^{(n+1)/2} 
$$
as $\lambda \rightarrow +\infty$.
\end{itemize}
\end{theorem}

\begin{remark}
When $\beta>\beta_c$, using the Mellin transform identity $\int_0^{+\infty} e^{-\tau\mu}\tau^{s-1}\,d\tau=\Gamma(s)\mu^{-s}$ and the definition of $Z_1(\tau)=\Tr(e^{-\tau P_1})$, we have
$$
A(\beta,n) = \gamma_n\,\zeta_{P_1}(d_H/2)
$$
where $\gamma_n$ is defined by \eqref{def_gamman} and $\zeta_{P_1}(s)=\sum_{j\geq 1}\mu_j^{-s}$ is the spectral zeta function of $P_1=-\partial_x^2+\frac{C_\beta}{x^2}+x^\beta$.
Moreover, as $\beta\downarrow \beta_c=2/n$, one has $d_H\downarrow n+1$, while the heat trace has the small-time asymptotics 
$Z_1(\tau)\sim (4\pi)^{-1/2}\frac{\Gamma(1/\beta)}{\beta}\,\tau^{-1/2-1/\beta}$,
which implies that $A(\beta,n)$ has the precise blow-up
$$
A(\beta,n)\sim \frac{K_n}{\beta-\beta_c},\qquad 
K_n = \frac{2}{\,n\,(4\pi)^{\frac{n+1}{2}}\,\Gamma\!\left(\frac{n+3}{2}\right)} 
= \frac{2}{n}\gamma_{n+1} = \frac{4}{n}C_n .
$$
For instance, $K_1=\frac{1}{2\pi}$, $K_2=\frac{1}{6\pi^2}$.
For the numerical evaluation at fixed $\beta>\beta_c$, one may compute the low spectrum $\{\mu_j\}$ of $P_1$ on $[0,L]$ with Dirichlet at $x=L$ and the Friedrichs (vanishing) behavior at $x=0$, then approximate $\zeta_{P_1}(d_H/2)$ by a partial sum $\sum_{j\leq M}\mu_j^{-d_H/2}$ with a tail controlled via the 1D Weyl law $N_1(\mu)\sim c_\beta\,\mu^{1/2+1/\beta}$.

\medskip
The Weyl asymptotics in Theorem~\ref{theo:A} are proved for each fixed $\beta$, and no uniform control of the remainder as
$\beta\downarrow\beta_c$ is claimed. Therefore the limits $\lambda\to+\infty$ and $\beta\downarrow\beta_c$ need not commute.
A convenient way to isolate the logarithmic contribution is to use the equivalent reformulation $\lambda^{d_H/2}\sim \lambda^{d_H/2}-\lambda^{(n+1)/2}$ that is valid for every fixed $\beta>\beta_c$ since $\lambda^{(n+1)/2}=\mathrm{o}(\lambda^{d_H/2})$ as $\lambda\to+\infty$.
Using $d_H/2=(n+1)/2+\frac{n}{4}(\beta-\beta_c)$, one has $\lambda^{d_H/2}-\lambda^{(n+1)/2}=\lambda^{(n+1)/2} (\lambda^{\frac{n}{4}(\beta-\beta_c)}-1 )$.
Together with $A(\beta,n)\sim K_n/(\beta-\beta_c)$ with $K_n=\frac{4}{n}C_n$, this yields
$$
\lim_{\beta\downarrow\beta_c} A(\beta,n)\big(\lambda^{d_H/2}-\lambda^{(n+1)/2}\big)
= C_n\,\lambda^{(n+1)/2}\ln\lambda,
$$
which explains how the logarithm arises in the critical case $\beta=\beta_c$ after cancelling the pole of $A(\beta,n)$.

We can do the same remark at the level of heat kernels. For $\delta>0$ one has
$\int_t^1 \tau^{-1+\delta}\,d\tau=\frac{1-t^\delta}{\delta}=-\ln t+\mathrm{O}(\delta(\ln t)^2)$ as $\delta\to0^+$, where we have used the lower cut-off $t\in(0,1)$ in the Mellin integral.
Thus, at the threshold the pole $1/\delta$ is converted into a $|\ln t|$ contribution after truncation.
In our setting the threshold corresponds to $\delta=s-\frac12-\frac1\beta=0$, i.e., $\beta=\beta_c$.
\end{remark}

\begin{remark}
When $M=\R/2\pi \Z $ and $X$ is diffeomorphic to the hemisphere, endowed with the so-called Grushin metric, the authors of \cite{Boscain} derived the Weyl law using an explicit computation of the spectrum. We recover their result as a particular case with $n=1$ and $\beta = \beta_c = 2$. 
\end{remark}

\begin{remark}
There exists an abundant literature on the asymptotic spectral study of singular metrics, of Laplace-Beltrami operators that degenerate at the boundary of the domain, of singular Schr\"odinger operators with potentials blowing up at the boundary, etc (see, e.g., \cite{Birman, Ivrii} and references therein).
An article that is close to ours, and that is complementary,
is the recent one \cite{ChitourPrandiRizzi}, in which the authors investigate the asymptotic behavior of the Weyl counting function on $d$-dimensional singular Riemannian manifolds, showing that, under some curvature and convexity assumptions, $N(\lambda)\sim \frac{1}{(4\pi)^{d/2}\Gamma(d/2+1)}\lambda^{d/2} v(\lambda)$ as $\lambda\rightarrow+\infty$, where the volume $v(\lambda)$ of the set at distance greater than $1/\sqrt{\lambda}$ from the boundary is assumed to be a slowly-varying function of $\lambda$. 
Following \cite[Section 7]{ChitourPrandiRizzi}, their result\footnote{It should be noted that, with the notations of \cite[Section 7]{ChitourPrandiRizzi}, it is actually not required that $m$ be an integer. Therefore, indeed, their result can be applied to the case of a real number $\beta\leq\beta_c$. We thank the authors for interesting discussions regarding this issue.}
covers our subcritical and critical cases, i.e., $\beta\leq\beta_c$, and indeed yields the same constants as in Theorem \ref{theo:A}. 
In the supercritical case $\beta>\beta_c$, the result of \cite{ChitourPrandiRizzi} only yields an asymptotic two-sided inequality $C_1\,\lambda ^{d_H/2} \leq N(\lambda)\leq C_2\,\lambda ^{d_H/2}$ (while Theorem \ref{theo:A} gives an equivalent). It is noticeable anyway that their result covers also cases where $\beta$ may vary along the boundary. 

After a first version of our article was submitted end of 2024, first of all, Bernard Helffer pointed out the reference \cite{Metivier_JEDP1975}, and second, an anonymous referee mentioned the existence of the article \cite{VulisSolomjak}, which we did not know (see also the older article \cite{Solomesh}).
We are very much indebted to B. Helffer and to the referee for these crucial indications. It is a fact that both papers \cite{Metivier_JEDP1975, VulisSolomjak} actually contain our Theorem \ref{theo:A} (to apply them, keep the metric in the form \eqref{initial_metric} and express the Dirichlet form). But, on the one part, the constants in \cite{Metivier_JEDP1975} are not as explicit as they are in our result; on the other part, in our paper we elaborate another proof, based on heat kernel considerations. Anyway, we think that the main interest of the present article is the relationship with the application to gas giants.

Beyond the fact that the three Weyl regimes (subcritical $\beta<\beta_c$, critical $\beta=\beta_c$, supercritical $\beta>\beta_c$, with $\beta_c=2/n$) are classically accessible through earlier works, the present paper contributes three features that, to the best of our knowledge, are not available in this level of explicitness:

(i) We provide a heat-kernel based derivation, yielding \emph{explicit} constants in all regimes, including the critical logarithmic case. 

(ii) We develop a unified reduction from the separable model to the general geometry via local quasi-isometries. 

(iii) We identify the associated \emph{Weyl measures} (see Section \ref{sec_weyl_measures}): in the subcritical regime the normalized volume on $X$, and in the supercritical regime a boundary-supported measure on $\partial X$, with boundary concentration for a density-one subsequence of eigenpairs. 
\end{remark}

\begin{remark}
As alluded to above, to prove Theorem~\ref{theo:A}, we make an intensive use of heat kernels. Alternatively, it is possible to use Dirichlet-Neumann bracketing (as in \cite{ChitourPrandiRizzi}). 
Both methods allow to treat conormal jump singularities of the metric $\bar g$ inside $X$ that model layering in the gas planet. 
For the proof using Dirichlet-Neumann bracketing, see \cite[Section 3]{DoktorarbeitCharlotte}.
\end{remark}

\begin{remark}\label{rem_nextterm}
A natural question is whether $\beta$ can be determined from the Weyl asymptotics. Indeed, when $n$ is known, $\beta$ can be determined in the case where $\beta \geq \beta_c$. When $\beta < \beta_c$ the question remains open. To shed light on this, it would be useful and interesting to get the next term in the small-time heat trace expansion (see Section~\ref{ss:J}) when $\beta \leq \beta_c$.
We refer to Section \ref{sec_open} for other open problems.
\end{remark}

\subsubsection{Weyl measures}\label{sec_weyl_measures}
We next compute the \emph{Weyl measures}, which are the probability measures $w_g$ on $X$, defined, if the limit exists, by
$$ 
\int_X f dw_g = \lim_{\lambda \rightarrow +\infty}  \frac{1}{N(\lambda)} \sum _{\lambda_j \leq \lambda} \int_X f\vert\phi_j\vert^2 \, dv_g 
$$
for any function $f : X \rightarrow \R$ that is continuous up to the boundary of $X$. 
Such measures have been introduced in \cite{CHT1,CHT3} in the framework of sub-Riemannian geometry in order to provide an account of how the high-frequency eigenfunctions concentrate. 

\begin{theorem}\label{theo:B} 
\begin{itemize}
\item If $\beta \geq \beta_c $ then the Weyl measure is $\delta_{x=0} \otimes dv_G / v_G(M)$.
\item If $\beta < \beta_c$ then the Weyl measure is the uniform probability distribution given by the normalized volume of $(X,g)$, that is $dv_g / v_g(X)$.
\end{itemize}
\end{theorem}

Using \cite[Corollary 7.1]{CHT3}, we obtain the following consequence.

\begin{corollary}\label{cor_qe}
If $\beta \geq \beta_c$ then there exists a density-one subsequence $(\phi_{j_k})_{k \in \N^*}$ of the sequence of eigenfunctions that concentrates on $\partial X$, meaning that for any compact subset $K\subset X \setminus \partial X$, we have
$$ 
\lim_{k \rightarrow + \infty} \int_K \vert\phi_{j_k}\vert^2 \, dv_g = 0 .
$$
\end{corollary}

\section{Proofs of Theorems \ref{theo:A} and \ref{theo:B}}

Our strategy of proof is the following. We first treat the separable case (Sections \ref{sec:1d} to \ref{sec:cclAsep}). As a preliminary, we perform in Section \ref{sec:1d} a spectral study of the 1D Schr\"odinger operator $P_\omega$ defined by \eqref{equ:p1}, deriving exponential estimates for truncated heat traces. Then, in Section \ref{sec:heatcone}, we estimate the small-time asymptotics of the truncated heat trace of $\triangle_g$, near the boundary (actually, on a cone); the three cases $\beta>\beta_c$ (supercritical), $\beta = \beta_c$ (critical), $\beta\leq\beta_c$ (subcritical), must be treated in different ways. In Section \ref{sec:che}, using a heat parametrix, we glue together the heat kernel near the boundary and the Riemannian heat kernel far from the boundary. Finally in Section \ref{sec:cclAsep} we prove Theorem \ref{theo:A} in the separable case.

In Section \ref{sec:gene}, we show how to pass from the separable to the general case by using the fact that the metric $g$ is quasi-isometric to a separable metric. In Section \ref{sec_proof_theo:B}, we prove Theorem \ref{theo:B}. Our approach uses again heat traces. 

\subsection{Spectral study of the 1D Schr\"odinger operators} \label{sec:1d}

We consider the family of Schr\"odinger operators,
$$ 
P_1 = -\partial_x^2 + q_{C,\beta}(x) 
$$
where $q_{C,\beta}(x) = C x^{-2} + x^\beta$, $C > 0$ and $\beta > 0$, acting on $L^2((0,+\infty),dx )$. The operators $P_1$ are essentially self-adjoint if and only if $C \geq 3/4$; when $C<3/4$ we consider the Friedrichs extension of $P_1$ with core $C_0^\infty ((0,+\infty))$ (see Appendix \ref{app:weyl}). The spectrum of $P_1$ is discrete; we denote it by $0 < \mu_1 \leq \mu_2 \leq  \mu_3 \leq  \cdots$. We derive precise semi-classical asymptotics for the associated truncated heat trace.

Let $\chi : [0,+\infty) \rightarrow [0,1] $ be a smooth nonincreasing function with $\chi \equiv 1$ on  $[0,a]$ with $a>0$  and $\chi' \leq 0$ everywhere. We note that $\chi \equiv 1$ is included. We define the corresponding truncated heat trace by
$$ 
\mathcal{Z}_{\chi }(\tau)= \Tr\left( e^{-\tau P_1}\chi \right) \qquad\forall \tau > 0 .
$$ 
Let $\gamma = \max(1/\beta,1/2)$. 

\begin{proposition} \label{prop:DN} 
Given any $0 < \tau \leq 1 $, we have
$$ 
\mathcal{Z}_{\chi}(\tau) = \frac{1}{\sqrt{4\pi \tau }} \int_0^{+\infty}  e^{-\tau x^\beta }\chi(x)  dx+ \mathrm{O}\left( {\tau}^{-\gamma} \right) 
$$ 
and for $\tau \geq 1$,
$$ 
\mathcal{Z}_{\chi} (\tau )=\mathrm{O}\left( e^{-\mu_1 \tau } \right) 
$$ 
uniformly with respect to $\chi$ in both cases.

The counting function $N_1(\mu) = \# \{j\in \N^*\ \mid\ \mu_j \leq \mu \}$ satisfies $N_1(\mu) \sim A \mu^{\frac{1}{2} +\frac{1}{\beta}}$ as $\mu\to+\infty$ with $A = \sqrt{\frac{2}{\pi}}\frac{1}{\beta}B(3/2,1+1/\beta)$ where $B$ is the Beta function. 
\end{proposition}

\begin{proof}[Proof of Proposition \ref{prop:DN}.]
We first establish an elementary lemma. We denote by $Q^N$ (resp., $Q^D$) the self-adjoint operator $-\partial_x^2$ on an interval of length $1$ with Neumann (resp., Dirichlet) boundary condition. 

\begin{lemma}\label{lem_DN}
For $0 < \tau \leq 1$ and $\star\in\{D,N\}$, we have $\Tr(\mathrm{exp}(-\tau Q^\star)) = (4\pi \tau)^{-1/2} + \mathrm{O}(1)$.
\end{lemma}

\begin{proof}[Proof of Lemma \ref{lem_DN}.]
The estimate does not depend on the chosen interval. The spectrum of $Q^N$ is $\{ n^2 \pi^2 \ \mid\ n \in \N \} $ and the spectrum of $Q^D$ is $\{ n^2 \pi^2 \ \mid\ n\in\N^* \}$. Hence both traces differ by $1$, and it suffices to prove the estimate for $Q^N$. Writing
$$ 
\Tr(\mathrm{exp}(-\tau Q^N  )) = \frac{1}{2} \Big( 1+ \sum _{n\in \Z}e^{-\tau n^2 \pi^2} \Big)
$$
and applying the Poisson summation formula gives the result. 
\end{proof}

We now prove the proposition. We first consider the case where $\tau \leq 1$. We are going to apply Dirichlet-Neumann bracketing with the decomposition $(0,+\infty) = \cup _{j=0}^{+\infty} J_k $ where the intervals $J_k$ are defined below.

Let $x_0$ be defined by $q_{C,\beta}(x_0) = \min q_{C,\beta}(x)$. Then $q_{C,\beta}(x) = C x^{-2}+x^\beta $ is increasing on $[x_0, +\infty)$. Let $J_k = [x_0+k, x_0+k+1]$ with $k\geq 1$ and $J_0 = ]0, x_0+1]$. We have the following estimates for the Dirichlet and Neumann heat traces $\mathcal{Z}^\star_{k,\chi} $ on $J_k$,
for $k \geq 1$.
We set $H^{\mathrm D/\mathrm N}_{J_k} = Q^{\mathrm D/\mathrm N}_{J_k}+q_{C,\beta}(x)$ with $Q^{\mathrm D/\mathrm N}_{J_k}=-\partial_x^2$.
On $J_k$ we have $q_{C,\beta}(x)\geq q(x_0+k)$, hence $H^{\mathrm N}_{J_k}\geq Q^{\mathrm N}_{J_k}+q_{C,\beta}(x_0+k)$ and thus $e^{-\tau H^{\mathrm N}_{J_k}}\leq e^{-\tau q_{C,\beta}(x_0+k)}e^{-\tau Q^{\mathrm N}_{J_k}}$.
Taking traces and using that $\Tr(e^{-\tau Q^{\mathrm N}_{J_k}})=\frac{1}{\sqrt{4\pi\tau}}+\mathrm{O}(1)$ yields
$$
\mathcal{Z}^N_{k,1}(\tau ) \leq \left( \frac{1}{\sqrt{4\pi \tau }}+\mathrm{O}(1) \right)e^{-\tau q_{C,\beta}(x_0+k)}\leq \left(\frac{1}{\sqrt{4\pi \tau }}+\mathrm{O}(1) \right) \int_{x_0+k-1}^{x_0+k} e^{-\tau q_{C,\beta}(x)}dx
$$
where the second inequality has been obtained because, since $q_{C,\beta}$ is increasing on $[x_0+k-1,x_0+k]$.

Now, since $\chi'\leq 0$, we have $\chi(x)\geq \chi(x_0+k+1)$ on $J_k$, and since $q_{C,\beta}$ is increasing, we have $q_{C,\beta}(x)\leq q_{C,\beta}(x_0+k+1)$.
Using Lemma \ref{lem_DN}, this yields
\begin{equation}\label{eq:D}
\begin{split}
\mathcal{Z}^D_{k,\chi}(\tau)=\Tr(\chi e^{-\tau H^{\mathrm D}_{J_k}})
& \geq \left(\frac{1}{\sqrt{4\pi\tau}}+\mathrm{O}(1)\right) e^{-\tau q_{C,\beta}(x_0+k+1)}\,\chi(x_0+k+1) \\
&\geq \left(\frac{1}{\sqrt{4\pi \tau}} + \mathrm{O}(1) \right) \int _{x_0+k+1}^{x_0+k+2} e^{-\tau q_{C,\beta}(x)}\chi(x) dx . 
\end{split}
\end{equation}
Furthermore, using $q_{C,\beta}\geq0$, we estimate $H^{\mathrm N}_{J_0}\geq Q^{\mathrm N}_{J_0}$, so by Lemma \ref{lem_DN}, we find
$\mathcal{Z}^D_{k,\chi}(\tau)=\mathrm{O}(\tau^{-1/2})$.
By Neumann bracketing (min-max principle), the spectrum of the full operator is bounded below by the spectrum of the orthogonal sum of Neumann problems on the blocks $J_k$; hence, summing the blockwise bounds (including the $k=0$ block which contributes $\mathrm{O}(\tau^{-1/2})$), we obtain
\begin{equation} \label{equ:neu}
\mathcal{Z}_1 (\tau )\leq \mathrm{O}\left(\frac{1}{\sqrt{\tau}}\right)  + \left(\frac{1}{\sqrt{4\pi \tau }}+\mathrm{O}(1) \right)
\int_{x_0}^{+\infty}  e^{-\tau q_{C,\beta}(x)}dx .
\end{equation}
Moreover,
$$
\left\vert \int _0^{+\infty} \!\Big( e^{-\tau q_{C,\beta}(x) } - e^{-\tau x^\beta } \Big) dx \right\vert
=\int_0^{+\infty} e^{-\tau x^\beta}\bigl(1-e^{-\tau C/x^2}\bigr)\,dx
=\mathrm{O}(1)\qquad(0<\tau\leq 1),
$$
so that
$$
\mathcal{Z}_1 (\tau )\leq \frac{1}{\sqrt{4\pi \tau }}  \int_0^{+\infty} e^{-\tau x^\beta} dx + \mathrm{O}\!\left( {\tau}^{-\gamma} \right) .
$$
%
Similarly, using \eqref{eq:D},
we obtain the lower bound
\begin{equation}\label{equ:dir}
\mathcal{Z}_\chi(\tau)\ \ge\ \frac{1}{\sqrt{4\pi\tau}}\int_0^{+\infty} e^{-\tau x^\beta}\,\chi(x)\,dx + \mathrm{O}(\tau^{-\gamma}).
\end{equation}

The same lower bound is valid when replacing $\chi$ by $1 - \chi$.
Indeed, $(1-\chi)$ is increasing. On each unit block $J_k=[x_0+k,x_0+k+1]$, $\min_{J_k}(1-\chi)=(1-\chi(x_0+k))$ and $\sup_{J_k}q_{C,\beta}=q_{C,\beta}(x_0+k+1)$. The same blockwise argument as for \eqref{eq:D} (operator monotonicity of the exponential and $\int_{J_k}K^{\mathrm D}_{J_k}=\frac{1}{\sqrt{4\pi\tau}}+O(1)$) yields
$$
\mathcal{Z}^D_{k,\,1-\chi}(\tau)\geq\Big(\frac{1}{\sqrt{4\pi\tau}}+\mathrm{O}(1)\Big)\,e^{-\tau q_{C,\beta}(x_0+k+1)}\,(1-\chi(x_0+k)).
$$
Summing over $k\geq 0$ gives
\begin{equation*}
\begin{split}
\mathcal{Z}^D_{1-\chi}(\tau)
&\geq \Big(\frac{1}{\sqrt{4\pi\tau}}+\mathrm{O}(1)\Big)\sum_{k\geq 0} e^{-\tau q_{C,\beta}(x_0+k+1)}\,(1-\chi(x_0+k)) \\
&\geq \Big(\frac{1}{\sqrt{4\pi\tau}}+O(1)\Big)\big(1+O(\tau)\big) \sum_{k\geq 0} e^{-\tau (x_0+k+1)^\beta}\,(1-\chi(x_0+k)).
\end{split}
\end{equation*}
where we have used that, for $x\geq x_0$, $e^{-\tau q_{C,\beta}(x)} \geq e^{-\tau x^\beta}\,e^{-\tau C/x_0^2} = (1+\mathrm{O}(\tau))\,e^{-\tau x^\beta}$.
Let $S_\tau=\sum_{k\geq 0} e^{-\tau (x_0+k+1)^\beta}\,(1-\chi(x_0+k))$. This is a Riemann sum for
$G_\tau=\int_{x_0}^{+\infty} e^{-\tau (x+1)^\beta}\,(1-\chi(x))\,dx$, and a standard estimate (by the mean value theorem on each block and summing the sup norms of $G_\tau'$) gives $S_\tau = G_\tau + \mathrm{O}(\vert\ln\tau\vert)$ as $\tau\to 0$.
Finally, $G_\tau=\int_{x_0}^{+\infty} e^{-\tau x^\beta}\,(1-\chi(x))\,dx + \mathrm{O}(1)$ because $G_\tau-\int_{x_0}^{+\infty} e^{-\tau x^\beta}(1-\chi(x))\,dx =-\int_{x_0}^{x_0+1} e^{-\tau x^\beta}(1-\chi(x))\,dx$.
Putting these bounds together yields
$$
\mathcal{Z}^D_{1-\chi}(\tau) \geq \frac{1}{\sqrt{4\pi\tau}} \int_0^{+\infty} e^{-\tau x^\beta}\,(1-\chi(x))\,dx
+ \mathrm{O}(\tau^{-\frac{1}{2}}\vert\ln\tau\vert).
$$
The above integral is of the order of $\tau^{-\frac{1}{2}-\frac{1}{\beta}}$, while the remainder $\tau^{-\frac{1}{2}}\vert\ln\tau\vert=\mathrm{o}(\tau^{-\frac{1}{2}-\frac{1}{\beta}})$ is of lower order than the main term $\sim \frac{1}{\sqrt{4\pi\tau}}\int_0^{+\infty} e^{-\tau x^\beta}(1-\chi)\,dx \sim \mathrm{const}\cdot \tau^{-\frac{1}{2}-\frac{1}{\beta}}$. Equivalently, for any $\gamma\in(\frac{1}{2},\frac{1}{2}+\frac{1}{\beta})$ we have
$$
\mathcal{Z}^D_{1-\chi}(\tau) \geq \frac{1}{\sqrt{4\pi\tau}}\int_0^{+\infty} e^{-\tau x^\beta}\,(1-\chi(x))\,dx + \mathrm{O}\big(\tau^{-\gamma}\big),
$$
which is the analogue of \eqref{equ:dir} with $\chi$ replaced by $1-\chi$.

\medskip

Now, we use a variant of the fact that
$$ 
(A+B = A'+B' ,\ A \geq A' ,\ B \geq B') \quad\Rightarrow\quad (A=A' ,\ B=B').
$$
We take $A = \mathcal{Z}_\chi$, $B = \mathcal{Z}_{1-\chi}$, $A' = (4\pi \tau)^{-\frac{1}{2}} \int_0^{+\infty} e^{-\tau x^\beta} \chi(x) dx$ and $B'= (4\pi \tau)^{-\frac{1}{2}} \int_0^{+\infty} e^{-\tau x^\beta} (1 - \chi (x)) dx$. By \eqref{equ:neu}, we have
$$ 
A+B = A'+B' + \mathrm{O}\left(\tau^{-\gamma}\right)
$$
and, from \eqref{equ:dir},
$$ 
A \geq A' +\mathrm{O}\left(\tau^{-\gamma}\right), ~  B \geq B' + \mathrm{O}\left(\tau^{-\gamma }\right) .
$$
It follows that $A = A' + \mathrm{O}\left(\tau^{-\gamma}\right)$ and $B = B' + \mathrm{O}\left(\tau^{-\gamma}\right)$.
In particular, 
$$ 
\mathcal{Z}_1(\tau) \sim \frac{1}{(4\pi \tau)^{\frac{1}{2}}} \int_0^{+\infty} e^{-\tau x^\beta} dx .
$$
Using the Karamata tauberian Theorem (recalled in Appendix~\ref{app:ka}), we get 
$$ 
N_1(\mu) \sim A \mu^{\frac{1}{2}+\frac{1}{\beta}} 
$$
as $\mu \rightarrow +\infty$, with the constant $A$ defined in Proposition~\ref{prop:DN}.

We now prove the exponential upper bound for $\tau \geq 1$. We note that $\mathcal{Z}_\chi(\tau) \leq \mathcal{Z}_{1}(\tau)$. Moreover all eigenvalues $\mu_j$ of $P_1$ are larger than the minimum $q(x_0)$ of $q$. Then, for $\tau \geq 1$, we have
$$  
\mathcal{Z}_{1}(\tau) = e^{-\tau \mu_1}\sum_{j=1}^{+\infty} e^{-\tau (\mu_j-\mu_1)}\leq e^{-\tau \mu_1}\sum_{j=1}^{+\infty} e^{- (\mu_j-\mu_1)} .
$$
The Weyl law applied to $P_1$ implies that the sum at the right-hand side converges and thus $\mathcal{Z}_\chi(\tau) \leq c e^{-\tau \mu_1}$ for some $c > 0$. 
\end{proof}

\begin{corollary}\label{coro:upper}
There exists $C_1>0$ only depending on $C$ and $\beta $ but not on $\chi $, such that, setting $J=\int_0^{+\infty} \chi(x) dx$, we have, for every $\tau>0$, 
$$
\mathcal{Z}_\chi(\tau) \leq \frac{C_1}{\sqrt{\tau}}\min \left( J, \tau^{-1/\beta} \right) .
$$
\end{corollary}

\begin{proof}
The  bound involving $J$ follows from the estimate $e(\tau,x,x)\leq (4\pi \tau )^{-\frac{1}{2}}$ which is valid for any positive potential $q$ (see \cite[Theorem 2.1.6]{Dav}). The other bound follows from Proposition \ref{prop:DN} and from the estimate $\int_0^{+\infty} e^{-\tau x^\beta }dx = \mathrm{O}\left(\tau ^{-1/\beta }\right)$. 
\end{proof}


\begin{proposition}\label{prop:unit} For any $\omega >0$, 
the operator $P_\omega $, defined by \eqref{equ:p1}, is unitarily equivalent to $\omega^{2/(2+\beta)}P_1$. In particular, the spectrum of $P_\omega$ is $\omega^{2/(2+\beta)}$ times the spectrum of $P_1$.
\end{proposition}

\begin{proof}
Considering the unitary map $U: L^2(\R^+, dx ) \rightarrow L^2 (\R^+, dx ) $ defined by
$$ 
U f(x) = \omega^{\frac{1}{2(2+\beta)}}f(\omega^{\frac{1}{2+\beta}}x) ,
$$
we find that $U^\star P_\omega U = \omega^{2/(2+\beta)} P_1$.
\end{proof}

\subsection{Truncated heat asymptotics for the cone $X_\infty $}\label{sec:heatcone}

In this subsection, we compute the small-time asymptotics of the truncated heat trace,
$$ 
Z_{\infty,\chi} (t)=\Tr\left(e^{-t\triangle_g} \chi \right)
$$
where $\chi $ is as in Section~\ref{sec:1d} and moreover is compactly supported in $[0,+\infty)$, and $g$ is the metric $g = dx^2+ x^{-\beta } g_1 $ on the cone $X_\infty= (0,+\infty) \times M$. The manifold $M$ is equipped with the metric $g_1$ that is independent of $x$. Here, we do not assume that $\partial M$ is empty: this will be useful in the proof of Theorem \ref{theo:B}. We will only use the Weyl asymptotics on $M$. 

Using the direct sum decomposition given in \eqref{equ:direct}, we have
\begin{equation} \label{equ:trace} 
Z_{\infty,\chi} (t) = \sum_{k=1}^{+\infty} \Tr\left(e^{-t P_{\omega_k}}\chi \right) 
\end{equation}
where we recall that the $\omega_k$ are the eigenvalues of $\triangle_M$ and $P_\omega$ is defined by \eqref{equ:p1}. 

We make the following two preliminary observations:
\begin{itemize}
\item For $k$ fixed and $t\rightarrow 0^+$, we have $ \Tr(e^{-t P_{\omega_k}}\chi )=\mathrm{O}(t^{-1/2})$. This term will be negligible in the sequel because the global trace is not less than $C/t$ for some $C > 0$ (since $\dim(X) \geq 2$).
\item For $t > 0$ fixed, the smooth function $f : \omega \mapsto \Tr(e^{-t P_{\omega}}\chi)$ has a fast decay at infinity: by Proposition~\ref{prop:unit},
$$ 
f(\omega) \leq \Tr(e^{-t P_{\omega}} )= \Tr(e^{-t \omega^{2/(2+\beta)} P_1} ) .
$$
The claim then follows from the second assertion given in Proposition~\ref{prop:DN}.
\end{itemize}

We split the sum \eqref{equ:trace} into two parts,
$$ 
Z_{\infty,\chi} (t) = \sum _{\omega_k <1} + \sum _{\omega_k \geq 1} = Z_{\infty,\chi}^0(t)+Z_{\infty,\chi}^1(t) .
$$
The first part, $Z_{\infty,\chi}^0(t)$, is $\mathrm{O}(t^{-1/2})$ by the first preliminary observation and we thus only have to estimate the second part, $Z_{\infty,\chi}^1(t)$. Using Proposition~\ref{prop:unit} and its proof, we have
$$ 
Z_{\infty,\chi}^1 (t) = \sum _{\omega_k \geq 1} \Tr\left( e^{-t\omega_k^{2/(2+\beta)}P_1}\chi (\cdot/\omega_k^{1/(2+\beta)}\right) = \sum_{\omega_k\geq 1} \mathcal{Z}_{\chi (\cdot/\omega_k^{1/(2+\beta)})}
(t\omega_k^{2/(2+\beta)}) 
$$
and we remark  that, for $\omega \geq 1$,  the function $\chi (\cdot/\omega^{1/(2+\beta)}) $ is identically equal to $1$ on $[0,a]$ (where $a$ was introduced Section \ref{sec:1d}), so that we can use the estimate of Section~\ref{sec:1d}.

Converting this sum into an integral (see Appendix~\ref{app:ka}), using the Weyl law on $M$, we obtain
$$ 
\# \{ \omega_k \leq \omega \}\sim \gamma_n \mathrm{Vol}(M) \omega^{n/2} \quad\textrm{as}\quad \omega \rightarrow +\infty . 
$$
Using Proposition~\ref{prop:kara-inv} and the definition of $f$, we get
$$ 
Z_{\infty,\chi}^1 (t) \sim \frac{n \gamma_n \mathrm{Vol}(M)}{2} \int_{1}^{+\infty} \mathcal{Z}_{\chi (./\omega^{1/(2+\beta)})}(t \omega^{2/(2+\beta)}) \omega^{\frac{n}{2}-1}\,d\omega .
$$
Making the change of variable $\tau = t \omega^{2/(2+\beta)}$, we arrive at the following lemma. Recall that $\delta_H=n(1+\frac{\beta}{2})$ (as defined after Proposition~\ref{lem_Hausdorff}).

\begin{lemma} \label{lemm:int}
The following holds,
$$
Z_{\infty,\chi}^1(t) \sim \frac{n \gamma_n(\beta +2) \mathrm{Vol}(M)}{4 t^{\delta_H/2}} \int_{t}^{+\infty} \mathcal{Z}_{\chi(.\sqrt{t/\tau})}(\tau) \tau^{\frac{\delta_H}{2} - 1} d\tau\ \text{as $t \rightarrow 0^+$ .}
$$
\end{lemma}

The  integral
\begin{equation} \label{equ:I}
I(t) = \int_{t} ^{+\infty} \mathcal{Z}_{\chi (.\sqrt{t/\tau})}(\tau) \tau^{\frac{\delta_H}{2} - 1} \, d\tau
\end{equation}
is convergent at $\tau = \infty$ for all $\beta > 0$ but, in general, not at $\tau =0$ because if $\beta \leq \beta _c$ then $ \frac{\delta_H}{2} -1 \leq \frac{1}{2} + \frac{1}{\beta} - 1$. We can compare this with the estimate in Corollary~\ref{coro:upper}. 

\subsubsection{Case $\beta > \beta_c$} \label{ss:larger}

We estimate the small-time behavior of $I(t)$ defined by \eqref{equ:I}. By the monotone convergence theorem, we have 
$$
\lim _{\varepsilon \rightarrow 0^+} \int_0^{+\infty} e(\tau,x,x) \chi (\varepsilon x) \, dx = \int_0^{+\infty} e(\tau,x,x) \, dx
$$
where $e$ is the heat kernel of $P_1$. Hence, for any $\tau > 0$,
$$ 
\lim _{t\rightarrow 0^+} \mathcal{Z}_{\chi(.\sqrt{t/\tau})}(\tau) = \mathcal{Z}_1 (\tau ) .
$$
Using, again, the monotone convergence theorem, we conclude that
$$ 
\lim _{t\rightarrow 0^+} I(t) = \int_0^{+\infty} \mathcal{Z}_{1}(\tau) \tau^{\frac{\delta_H}{2} - 1} \, d\tau . 
$$
From Corollary~\ref{coro:upper}, we get 
$$ 
\mathcal{Z}_1(\tau) \tau^{\delta_H/2 -1} \leq C \tau^{\frac{n}{2}(1+\frac{\beta}{2})-1-\frac{1}{2} -\frac{1}{\beta}} e^{-\mu_1 \tau}  
$$
and $\beta > \beta_c = 2/n$ implies $\frac{n}{2} (1 + \frac{\beta}{2})-1-\frac{1}{2} - \frac{1}{\beta} > -1$. Thus, the corresponding limit is finite:
$$ 
\lim _{t\rightarrow 0^+} I(t) = \int_0^{+\infty} \mathcal{Z}_{1}(\tau) \tau^{\frac{\delta_H}{2} - 1} \, d\tau < +\infty . 
$$

\subsubsection{Case $\beta \leq \beta_c$} \label{ss:J}

By the second estimate in Proposition \ref{prop:DN}, the contribution to the integral from $1$ to $+\infty $ in the integral $I(t)$ defined by \eqref{equ:I} is $\mathrm{O}(1)$ uniformly with respect to $t$ and hence the corresponding part of $Z_{\infty,\chi}^1(t)$ is $\mathrm{O}(t^{-\delta_H/2})$ as $t\rightarrow 0^+$, which will be negligible. We only need to estimate the small-time asymptotics of 
$$ 
I_1(t) = \int_t^1 \mathcal{Z}_{ \chi (\cdot\,\sqrt{t/\tau})}(\tau) \tau^{\frac{\delta_H}{2} - 1} \, d\tau .
$$

\paragraph{Sub-case $\beta < \beta_c$.}
We prove that there exists a $\delta > 0$ such that
\begin{equation}\label{eq:Jbetasmall}
\Tr\left( e^{-t\triangle_g} \chi \right) = \mathrm{O}\big(J^\delta t^{-(n+1)/2} \big) 
\end{equation}
as $t\rightarrow 0^+$. We split the integral 
$$
I_1(t) = \int_t^1 \Tr\left( e^{-\tau P_1} \chi \big( \cdot\, \sqrt{t/\tau} \big)\right) \tau^{\frac{\delta_H}{2} - 1} d\tau = I_{1,1}(t) + I_{1,2}(t) = \int_t^{\tau_0} + \int_{\tau_0}^1
$$
where $\tau _0$ satisfies $\tau_0^{-1/\beta} = J \sqrt{\tau_0/t}$, i.e., $\tau_0 = \left(t/J^2 \right)^{\beta/(2+\beta)}$ with $J = \int_0^{+\infty} \chi(x) \, dx$ as in Corollary~\ref{coro:upper}. We get upper bounds for $I_{1,1}(t)$ and $I_{1,2}(t)$ using the upper bounds given in Corollary~\ref{coro:upper} as follows. Using the first argument in the minimum, we have
$$
I_{1,1}(t) \leq C \frac{J}{\sqrt{t}}\tau_0 ^{\delta_H/2} = C J^{1-n\beta/2} t^{n\beta/4-1/2}
$$
Similarly, using the second argument in the minimum, we find that
$$ 
I_{1,2}(t) \leq C \int_{\tau_0}^1 \tau^{-1/\beta} \tau ^{\delta_H/2-3/2} \, d\tau \leq C J^{1 - n\beta/2} t^{n\beta/4-1/2} . 
$$
Finally,
$$ 
t^{-\delta_H/2} I(t) \leq C J^{1 - n\beta/2} t^{-(n+1)/2} 
$$
so that we can take $\delta = 1 - n \beta/2$. We will use this further in Section~\ref{sec:cclAsep} by choosing $J$ small.

\paragraph{Sub-case $\beta = \beta_c$.}
When $\beta = 2/n$, we have to estimate the asymptotics of
$$ 
I_1(t) = \int_t^1 \Tr\left( e^{-\tau P_1 } \chi(\cdot\sqrt{t/\tau}) \right) \tau ^{(n-1)/2} \, d\tau .
$$
Using the estimate of Proposition~\ref{prop:DN}, we get 
$$ 
I_1(t) \sim \frac{1}{\sqrt{4\pi}} \int _{t}^1 \tau^{\frac{n}{2}-1} \, d\tau \int _0^{+\infty} e^{-\tau x^{2/n}}\chi( x\sqrt{t/\tau}) \, dx  
$$
modulo terms of smaller order in $\tau$. Using the change of variable $y = \tau x^{2/n}$, we get
$$ 
I_1(t) \sim \frac{n}{4\sqrt{\pi}}\int_t^1 \frac{d\tau}{\tau } \int_0^{+\infty}e^{-y}y^{\frac{n}{2}-1}\chi \left(\sqrt{t}\tau^{-(n+1)/2}y^{n/2} \right) \, dy 
= \frac{n}{4\sqrt{\pi}}\int_t^1 \frac{d\tau}{\tau } F\left(\tau^{(n+1)/2}/t^{1/2}\right)
$$
where the function $F$, defined by
$$ 
F(X)= \int_0^{+\infty} e^{-y} y^{\frac{n}{2}-1}\chi \left(y^{n/2}/X \right) \, dy ,
$$
is smooth, bounded and monotone increasing on $(0,\infty)$ with $\lim_{X\to+\infty}F(X)=\Gamma(n/2)$.
Indeed, $\chi$ is decreasing and $X\mapsto y^{n/2}/X$ is decreasing, hence $X\mapsto \chi(y^{n/2}/X)$ is increasing; dominated convergence then yields the limit.
Besides, we have $F(X) = \mathrm{O}_{X \rightarrow 0} \left( X \right)$. 
Indeed, making the change of variables $z=y^{n/2}/X$, we have
$F(X)=\frac{2}{n} X\int_0^{+\infty} e^{-(Xz)^{2/n}}\chi(z)\,dz$.
Since $0\leq\chi\leq 1$ and $\int_0^{+\infty}\chi(z)\,dz<+\infty$ (by construction of the cutoff), the dominated convergence theorem gives
$F(X)=\frac{2}{n} X\int_0^{+\infty} \chi(z)\,dz + \mathrm{o}(X)$ as $X\to 0$, whence the result.

Using the new variable $u= \tau^{(n+1)/2}/t^{1/2}$, we get
$$ 
I_1(t) \sim \frac{n}{2(n+1)\sqrt{\pi}}\int_{t^{n/2}}^{t^{-1/2}} F(u) \, \frac{du}{u} .
$$
For any $0<\varepsilon<\Gamma(n/2)$, define $U_\varepsilon$ as the positive real number such that $\Gamma(n/2)-F(U_\varepsilon)=\varepsilon$. Now, for any $t>0$ such that $t^{n/2}\leq 1$ and $1\leq U_\varepsilon\leq t^{-1/2}$, 
we split the integral into three pieces:
$$
\int_{t^{n/2}}^{t^{-1/2}}\frac{F(u)}{u}\,du
=\int_{t^{n/2}}^{1}\frac{F(u)}{u}\,du+\int_{1}^{U_\varepsilon}\frac{F(u)}{u}\,du
+\int_{U_\varepsilon}^{t^{-1/2}}\frac{F(u)}{u}\,du.
$$
Since $F(u)/u=\mathrm{O}(1)$ near $0$, we have $\int_{t^{n/2}}^{1}\frac{F(u)}{u}\,du=\mathrm{O}(1)$ uniformly as $t\rightarrow 0^+$; the middle piece is a fixed constant. For the last piece,
$$
\int_{U_\varepsilon}^{t^{-1/2}}\frac{F(u)}{u}\,du
=\Gamma\!\left(\frac{n}{2}\right)\ln\!\Big(\frac{t^{-1/2}}{U_\varepsilon}\Big)
+\int_{U_\varepsilon}^{t^{-1/2}}\frac{F(u)-\Gamma(\frac{n}{2})}{u}\,du,
$$
and the remainder is bounded by $\varepsilon\ln(t^{-1/2})=\frac{\varepsilon}{2}\,|\ln t|$. Since $\varepsilon$ is arbitrary, we obtain
$$
\int_{t^{n/2}}^{t^{-1/2}}\frac{F(u)}{u}\,du
=\frac{1}{2}\,\Gamma\!\left(\frac{n}{2}\right)\,|\ln t| + \mathrm{O}(1)
$$
as $t\rightarrow 0^+$.
Therefore,
$$
I_1(t)\;\sim\;\frac{n}{2(n+1)\sqrt{\pi}} \frac{1}{2}\,\Gamma\!\left(\frac{n}{2}\right)\,|\ln t|
=\frac{n\,\Gamma(\frac{n}{2})}{4(n+1)\sqrt{\pi}}\;|\ln t|
$$
as $t\rightarrow 0^+$.

\subsection{The heat parametrix in the separable metric case}\label{sec:che}
We adapt the method of \cite{Che}; the off-diagonal heat kernel estimates that we need are established in Lemma~\ref{lem_offdiag} below. We denote by $z,z'$ some generic points of $X$ and by $z=(x,m)$, $z'=(x',m')$ generic points of $[0,1) \times M \subset X$. Let $\chi$ be as in the previous sections, vanishing near $x=1$ and extended by $0$ inside $X$. Let $\eta \in C_0^\infty ([0, 1))$ so that $\eta =1$ near the support of $\chi $, and $\eta_0 \in C_0^\infty (X)$, vanishing near $\partial X$ and equal to $1$ near the support of $1-\chi$. We choose  $a > 0$ so that $\eta_0$ vanishes for $x\leq 2a$.
In the next result, we show that
$$ 
p(t,z,z')= \eta(x) e_\infty (t,z,z') \chi(x') + \eta_0 (x) e_0 (t,z,z') (1-\chi (x') )
$$
where $e_\infty $ is the heat kernel on the cone $X_\infty$ and $e_0 $ the Riemannian heat kernel generated by the Laplacian $\triangle_g$ on $X\setminus \{ x \leq a \} $ with  Dirichlet boundary conditions, 
is a good approximation of the heat kernel on $X$ as $t\rightarrow 0$.

\begin{remark}\label{rem_gap}
In the published version of this article (Ann. Henri Poincar\'e, 2026, \url{https://doi.org/10.1007/s00023-026-01678-z}), the proof of Proposition~\ref{prop_parametrix} below relied on the interior estimates of Appendix~\ref{app:loc}. As kindly pointed out to us by Luc Hillairet, that argument has a gap: Lemma~\ref{lem:localheat} yields estimates that are locally uniform on the compact subsets of $(X\setminus\partial X)\times(X\setminus\partial X)$ minus the diagonal, but not uniformly as one of the two points approaches $\partial X$, while the error kernel $r$ of the parametrix (see the proof below) involves points $z'$ arbitrarily close to $\partial X$. Note also that, when $\beta\geq\beta_c$, the volume of any neighborhood of $\partial X$ is infinite, so that even pointwise bounds on $r$, uniform up to the boundary, would not suffice to control its trace norm. In the present version, we fill this gap by establishing off-diagonal heat kernel estimates, in trace norm, that are uniform up to the boundary (Lemma~\ref{lem_offdiag}); the key ingredient is the Davies-Gaffney inequality, which is valid for the Friedrichs extension without any assumption at the boundary. We stress that the gap concerned the proof only: the statement of Proposition~\ref{prop_parametrix} is unchanged (the new proof even establishes the slightly stronger trace norm estimate $\Vert P(t)-e^{-t\triangle_g}\Vert_1=\mathrm{O}(t^\infty)$), and no other result of the paper is affected. Recall also that Theorem~\ref{theo:A} can alternatively be proved by Dirichlet-Neumann bracketing (see \cite[Section 3]{DoktorarbeitCharlotte}) and is covered as well by \cite{Metivier_JEDP1975, VulisSolomjak}.
\end{remark}

\begin{lemma}\label{lem_offdiag}
Let $(Y,h)$ be a smooth Riemannian manifold, not necessarily complete, and let $\triangle^Y$ be the Friedrichs extension of the Laplace-Beltrami operator of $(Y,h)$ with core $C_0^\infty(Y)$. Let $A$ be a differential operator of order $\leq 1$ on $Y$, with smooth coefficients supported in a compact set $K\subset Y$. Let $B\subset Y$ be measurable, and assume that there exist $\delta>0$ and a bounded Lipschitz function $\psi:Y\rightarrow[0,+\infty)$, with $\vert\nabla\psi\vert\leq 1$ almost everywhere, such that $\psi=0$ on $B$ and $\psi\geq\delta$ on an open neighborhood $V$ of $K$ having compact closure in $Y$. Then, for every measurable function $\widetilde\chi$ supported in $B$ with $\vert\widetilde\chi\vert\leq 1$ (also denoting by $\widetilde\chi$ the corresponding multiplication operator), we have
$$
\Vert A\, e^{-s\triangle^Y} \widetilde\chi \Vert_1 = \mathrm{O}(s^N)\qquad\textrm{as}\quad s\rightarrow 0^+,\qquad \forall N>0 ,
$$
where $\Vert\cdot\Vert_1$ is the trace norm on $L^2(Y)$.
\end{lemma}

\begin{proof}
Let $f\in L^2(Y)$ and let $u=e^{-s\triangle^Y}(\widetilde\chi f)$. We fix open sets $S$ and $S'$ such that $K\subset S$, $\overline{S}\subset S'$ and $\overline{S'}\subset V$; the Sobolev norms used below are the usual ones on these sets, on which the metric is smooth and nondegenerate.

First, we have the Davies-Gaffney inequality
\begin{equation}\label{ineq_DG}
\Vert u\Vert_{L^2(\{\psi\geq\delta\})}\leq e^{-\delta^2/4s}\, \Vert f\Vert_{L^2(Y)} ,
\end{equation}
which is valid for the Friedrichs extension on any Riemannian manifold, with no completeness assumption and no assumption at the (possibly missing) boundary (see \cite{Davies1992, Grigoryan1994}). Let us recall the short argument. Since $\psi$ is bounded and Lipschitz, multiplication by $e^{\alpha\psi}$ (with $\alpha>0$) preserves the form domain of $\triangle^Y$, which is the closure of $C_0^\infty(Y)$ in $H^1$ norm. Hence, setting $J(t)=\int_Y \vert e^{-t\triangle^Y}(\widetilde\chi f)\vert^2\, e^{\alpha\psi}\,dv_h$, the integration by parts defining the Friedrichs extension and the Young inequality $2\alpha\vert v\vert\,\vert\nabla v\vert\leq 2\vert\nabla v\vert^2+\frac{\alpha^2}{2}\vert v\vert^2$ give $J'(t)\leq\frac{\alpha^2}{2}J(t)$ for every $t>0$. Since $\psi=0$ on $\mathrm{supp}\,\widetilde\chi$, we have $J(0^+)\leq\Vert f\Vert_{L^2}^2$, hence $\Vert u\Vert_{L^2(\{\psi\geq\delta\})}^2\leq e^{-\alpha\delta}J(s)\leq e^{-\alpha\delta+\alpha^2 s/2}\,\Vert f\Vert_{L^2}^2$, and \eqref{ineq_DG} follows by taking $\alpha=\delta/s$.

Second, by the spectral theorem, $\Vert(\triangle^Y)^j u\Vert_{L^2(Y)}\leq (j/e)^j\, s^{-j}\,\Vert f\Vert_{L^2}$ for every $j\in\N$, and thus, by interior elliptic regularity on the compact set $\overline{S'}$, we have $\Vert u\Vert_{H^{2k}(S')}\leq C\, s^{-k}\,\Vert f\Vert_{L^2}$ for every $k\in\N$. Combining with \eqref{ineq_DG} (note that $S'\subset V\subset\{\psi\geq\delta\}$) and with the interpolation inequality $\Vert\varphi u\Vert_{H^k}\leq C\,\Vert\varphi u\Vert_{H^{2k}}^{1/2}\,\Vert\varphi u\Vert_{L^2}^{1/2}$, where $\varphi\in C_0^\infty(S')$ is such that $\varphi=1$ on $S$, we infer that
$$
\Vert e^{-s\triangle^Y}\,\widetilde\chi\, \Vert_{L^2(Y)\rightarrow H^k(S)}\leq C\, s^{-k/2}\, e^{-\delta^2/8s} = \mathrm{O}(s^N)\qquad\forall k\in\N\qquad\forall N>0 .
$$

Third, we set $d=\dim Y$ and $m=d+1$. Since the coefficients of $A$ are supported in $K\subset S$ and since differential operators are local, we have the exact factorization $A\,e^{-s\triangle^Y}\,\widetilde\chi=\jmath\circ\iota\circ A\circ T(s)$, where $T(s)$ is the operator $e^{-s\triangle^Y}\widetilde\chi$, viewed as a map from $L^2(Y)$ to $H^{m+1}(S)$, where $A:H^{m+1}(S)\rightarrow H^m(S)$ is bounded, where $\iota:H^m(S)\hookrightarrow L^2(S)$ is the Sobolev embedding, and where $\jmath:L^2(S)\rightarrow L^2(Y)$ is the extension by zero. Since $m>d$, the embedding $\iota$ is trace class: its singular values are $\mathrm{O}(j^{-m/d})$ (see \cite{Birman}). Hence
$$
\Vert A\,e^{-s\triangle^Y}\,\widetilde\chi\,\Vert_1\leq\Vert\iota\Vert_1\,\Vert A\Vert_{H^{m+1}(S)\rightarrow H^m(S)}\,\Vert T(s)\Vert_{L^2(Y)\rightarrow H^{m+1}(S)} = \mathrm{O}(s^N)\qquad\forall N>0 . \qedhere
$$
\end{proof}

\begin{proposition}\label{prop_parametrix}
Let $P(t)$ be the operator of Schwartz kernel $p(t,\cdot,\cdot)$. We have
$$
\Tr \left( P(t) - e^{-t\triangle_g} \right) = \mathrm{O}\left(t^\infty \right) 
$$
as $t \rightarrow 0^+$.
\end{proposition}

\begin{proof}
Throughout the proof, functions supported in the collar $(0,1)\times M$ are regarded as functions on $X$, on $X_\infty$ or on $\{x>a\}$, with the corresponding identifications for operators. We denote by $\triangle_0$ the Dirichlet Laplacian on $\{x>a\}$, so that
$$
P(t) = \eta\, e^{-t\triangle_\infty}\,\chi + \eta_0\, e^{-t\triangle_0}\,(1-\chi) .
$$
Since $(\partial_t+\triangle)e_\infty=0$ and $(\partial_t+\triangle)e_0=0$, and since the differential expression of $\triangle_g$ coincides with that of $\triangle_\infty$ on $\mathrm{supp}(\eta)$ and with that of $\triangle_0$ on $\mathrm{supp}(\eta_0)$, the operator $R(t)$ of Schwartz kernel $r(t,z,z')=\left(\partial_t+(\triangle_g)_z\right)p(t,z,z')$ is
$$
R(t) = [\triangle,\eta]\, e^{-t\triangle_\infty}\,\chi + [\triangle,\eta_0]\, e^{-t\triangle_0}\,(1-\chi)
$$
where $[\triangle,\theta]v=(\triangle\theta)v-2\langle\nabla\theta,\nabla v\rangle_g$ is a differential operator of order one, with smooth coefficients supported in $\mathrm{supp}\,d\theta$.

Let us check that Lemma~\ref{lem_offdiag} applies to both terms of $R(s)$, yielding $\Vert R(s)\Vert_1=\mathrm{O}(s^N)$ for every $N>0$. By construction of the cutoffs, there exist $0<a_1<a_2<1$ such that $\mathrm{supp}(\chi)\subset\{x\leq a_1\}$ and $\mathrm{supp}(d\eta)\subset\{x> a_2\}$, and there exist $2a\leq c_1<c_0$ such that $\chi=1$ on $\{x\leq c_0\}$ and $\mathrm{supp}(d\eta_0)\subset\{2a\leq x<c_1\}$. On the collar, $g\geq dx^2$, so that any function $\psi$ of $x$ with $\vert\psi'\vert\leq 1$ satisfies $\vert\nabla\psi\vert_g\leq 1$. For the first term, we apply Lemma~\ref{lem_offdiag} with $Y=X_\infty$, $A=[\triangle,\eta]$, $K=\mathrm{supp}(d\eta)$, $\widetilde\chi=\chi$, $B=\{x\leq a_1\}$, $\psi=\min(\max(x-a_1,0),a_2-a_1)$, $V=\{a_2<x<1\}\times M$ and $\delta=a_2-a_1$. The resulting estimate is uniform with respect to the point of $B$, no matter how close to $\partial X$. For the second term, we apply Lemma~\ref{lem_offdiag} with $Y=\{x>a\}$, $A=[\triangle,\eta_0]$, $K=\mathrm{supp}(d\eta_0)$, $\widetilde\chi=1-\chi$, $B=\mathrm{supp}(1-\chi)$, $\psi=\min(\max(c_0-x,0),c_0-c_1)$ on the collar, extended by $0$ outside, $V=\{3a/2<x<c_1\}\times M$ and $\delta=c_0-c_1$.

Next, given any $u\in D(\triangle_\infty)$, the function $\eta u$, extended by $0$, belongs to $D(\triangle_g)$, with $\triangle_g(\eta u)=\eta\,\triangle_\infty u+[\triangle,\eta]u$. Indeed, $\eta u$ belongs to the form domain of $\triangle_g$ (approximate $u$ by elements of $C_0^\infty(X_\infty)$ in the form norm and multiply by $\eta$), its distributional Laplacian $\eta\,\triangle_\infty u+[\triangle,\eta]u$ belongs to $L^2(X,dv_g)$, and, for the Friedrichs extension, these two facts imply that $\eta u\in D(\triangle_g)$ (see \cite[Theorem~X.23]{RSII}). The same holds for $(\eta_0,\triangle_0)$. Hence, given any $f\in L^2(X,dv_g)$ and any $t\in(0,1]$, the map $s\mapsto e^{-(t-s)\triangle_g}P(s)f$ is of class $C^1$ on $(0,t)$, with
$$
\frac{d}{ds}\left( e^{-(t-s)\triangle_g}\,P(s)f\right) = e^{-(t-s)\triangle_g}\left( \triangle_g P(s)+P'(s)\right) f = e^{-(t-s)\triangle_g}\, R(s)\, f .
$$
Since $P(s)\rightarrow\mathrm{id}$ strongly as $s\rightarrow 0^+$ (because $\eta\chi=\chi$ and $\eta_0(1-\chi)=1-\chi$), we have the Duhamel formula
$P(t)-e^{-t\triangle_g} = \int_0^t e^{-(t-s)\triangle_g}\, R(s)\, ds$,
where the integral converges in trace norm, with $\Vert e^{-(t-s)\triangle_g}R(s)\Vert_1\leq\Vert R(s)\Vert_1$. Therefore
$$
\Vert P(t)-e^{-t\triangle_g}\Vert_1 \leq \int_0^t \Vert R(s)\Vert_1\, ds = \mathrm{O}(t^\infty) .
$$
Finally, $P(t)$ is trace class. Indeed, each of its two terms is the product of two Hilbert-Schmidt operators, for instance $\eta\,e^{-t\triangle_\infty}\,\chi=\left(\eta\,e^{-(t/2)\triangle_\infty}\right)\left( e^{-(t/2)\triangle_\infty}\,\chi\right)$, with 
$$
\Vert \eta\,e^{-(t/2)\triangle_\infty}\Vert_{\mathrm{HS}}^2=\int_{X_\infty}\eta^2\, e_\infty(t,z,z)\,dv_g(z)\leq Z_{\infty,\widetilde\chi}(t)<+\infty
$$
for any cutoff $\widetilde\chi\geq\eta^2$ as in Section~\ref{sec:heatcone}, and similarly for the other factors. Hence $e^{-t\triangle_g}$ is trace class as well, and $\vert \Tr \left( P(t)-e^{-t\triangle_g} \right) \vert\leq\Vert P(t)-e^{-t\triangle_g}\Vert_1=\mathrm{O}(t^\infty)$.
\end{proof}

\subsection{Completion of the proof of Theorem~\ref{theo:A} in the separable metric case}\label{sec:cclAsep}

\begin{lemma}[Interior contribution]\label{lem:interior}
We have
$$
\Tr\big(\eta_0 e_0(t)\,(1-\chi)\big)
= (4\pi t)^{-\frac{n+1}{2}}\!\int_X (1-\chi)\,dv_g + \mathrm{O}\big(t^{-\frac{n-1}{2}}\big)
$$
as $t\rightarrow 0^+$.
\end{lemma}

\begin{proof}
On $\mathrm{supp}(1-\chi)$ the metric is smooth up to the boundary and the diagonal heat kernel admits the classical expansion $e_0(t;z,z)\sim (4\pi t)^{-(n+1)/2}\big(1+a_1(z)t+a_2(z)t^2+\cdots\big)$ uniformly in $z$. Multiplying by the smooth cutoff $\eta_0(1-\chi)$ and integrating gives the stated expansion. 
\end{proof}

By the definition of $P(t)$ and by Proposition~\ref{prop_parametrix}, we get
\begin{equation*}
\Tr\big(e^{-t\triangle_g}\big)
= \underbrace{\Tr\big(\eta_0 e_0(t)\,(1-\chi)\big)}_{\text{interior}}  + 
\underbrace{\Tr\big(\eta\,e_\infty(t)\,\chi\big)}_{\text{truncated cone}}
+ \mathrm{O}(t^\infty).
\end{equation*}
as $t\rightarrow 0^+$.
Here, by cyclicity of the trace (each operator being a product of two Hilbert-Schmidt factors, as in the proof of Proposition~\ref{prop_parametrix}) and since $\eta\chi=\chi$ and $\eta_0(1-\chi)=1-\chi$, we have $\Tr\big(\eta\,e_\infty(t)\,\chi\big)=Z_{\infty,\chi}(t)$ and $\Tr\big(\eta_0\,e_0(t)\,(1-\chi)\big)=\int_X (1-\chi)\, e_0(t,z,z)\,dv_g(z)$.
The interior term is tackled by Lemma~\ref{lem:interior}. 
The truncated-cone term when $\beta\geq\beta_c$ is exactly the outcome of Section \ref{sec:heatcone} (the power $t^{-d_H/2}$ for $\beta>\beta_c$, and the coefficient of $|\ln t|$ at $\beta=\beta_c$). 
We obtain that:
\begin{itemize}
\item If $\beta>\beta_c$,
$$
\Tr\big(\eta\,e_\infty(t)\,\chi\big)
= A(\beta,n)\,t^{-\frac{d_H}{2}}\;\mathrm{Vol}_G(M) + \mathrm{o}\big(t^{-\frac{d_H}{2}}\big),
$$
with $A(\beta,n)$ as in Theorem~\ref{theo:A} and $G=g_1(0)$. In particular $t^{-\frac{d_H}{2}}\gg t^{-\frac{n+1}{2}}$, so the cone term dominates and Theorem~\ref{theo:A} follows.
\item If $\beta=\beta_c$,
$$
\Tr\big(\eta\,e_\infty(t)\,\chi\big)
= (4\pi t)^{-\frac{n+1}{2}} \frac{n\,\Gamma(\frac{n}{2})}{4(n+1)\sqrt{\pi}}\, |\ln t| \, \mathrm{Vol}_G(M) + \mathrm{O}\big(t^{-\frac{n+1}{2}}\big),
$$
so the logarithmic factor makes the boundary term dominate the interior one (see Lemma~\ref{lem:interior}), which has the same homogeneity but no $\ln$.
\end{itemize}
If $\beta<\beta_c$ then, by \eqref{eq:Jbetasmall}, for any cutoff $\chi$ as above,
$$
\Tr\big(\eta\,e_\infty(t)\,\chi\big) = Z_{\infty,\chi}(t) = \mathrm{O}\big( J^{1-n\beta/2}\, t^{-\frac{n+1}{2}} \big)
\qquad\textrm{where}\quad J=\int_0^{+\infty}\chi(x)\,dx .
$$
Hence, by Lemma~\ref{lem:interior},
$$
(4\pi t)^{\frac{n+1}{2}}\,\Tr\big( e^{-t\triangle_g} \big) = \int_X(1-\chi)\,dv_g + \mathrm{O}\big( J^{1-n\beta/2} \big) + \mathrm{O}(t)
$$
as $t\rightarrow 0^+$. Since, moreover, $0\leq v_g(X)-\int_X(1-\chi)\,dv_g=\int_X\chi\,dv_g\leq C\, J^{1-n\beta/2}$ (because the volume density on the collar is $x^{-n\beta/2}\,\vert dx\vert\,dv_1$: split the integral at $x=J$ and use $0\leq\chi\leq 1$), letting $t\rightarrow 0^+$ and then $J\rightarrow 0$, i.e., concentrating the support of $\chi$ near $\{x=0\}$, yields $\Tr\big(e^{-t\triangle_g}\big)\sim(4\pi t)^{-\frac{n+1}{2}}\,v_g(X)$, hence the subcritical case of Theorem~\ref{theo:A}.

\subsection{From the separable to the general case}\label{sec:gene}
In Sections \ref{sec:heatcone}--\ref{sec:cclAsep} we established the small-time heat trace asymptotics (hence the Weyl law) for a separable model metric in a neighborhood of $\partial X$. In this section, we explain how to transfer these asymptotics to a general metric $g$. 

\begin{lemma}\label{lem:qis}
For every $\varepsilon\in(0,1)$ there exist $\delta\in(0,1)$ and $\eta\in C^\infty_0([0,\delta))$ with $\eta\equiv1$ near $u=0$, such that the metric
$$
g_s = \eta(u)\,u^{-\alpha}(du^2+g_0(0)) + (1-\eta(u))g
$$
satisfies, on $\{0<u<\delta\}$,
\begin{equation}\label{eq:metriceq}
(1-\varepsilon)\,|v|_g^2 \leq |v|_{g_s}^2 \leq (1+\varepsilon)\,|v|_g^2,\qquad
(1-\varepsilon)\,dv_g \leq dv_{g_s} \leq (1+\varepsilon)\,dv_g.
\end{equation}
The Dirichlet forms are comparable:
\begin{equation}\label{eq:nablaeq}
(1-\varepsilon) \int|\nabla u|_g^2\,dv_g \leq \int |\nabla u|_{g_s}^2\,dv_{g_s} \leq (1+\varepsilon) \int|\nabla u|_g^2\,dv_g,\qquad \forall u\in C_c^\infty(X),
\end{equation}
and $g_s$ agrees with $g$ on $\{u\geq\delta\}$. Moreover, $g_s$ has the same degeneracy exponent $\alpha$ (equivalently $\beta$) and the same boundary metric $G=g_0(0)$ as $g$.
\end{lemma}

\begin{proof}
By continuity of $g_0(u)$ in $u$, choose $\delta>0$ such that $|g_0(u)-g_0(0)|\leq \varepsilon\big(du^2+g_0(u)\big)$ for $u\in[0,\delta]$. The stated inequalities follow by inspection from the definition of $g_s$, and the volume-form comparison is a consequence of the determinant estimate for small tensor perturbations.
\end{proof}

\begin{corollary}\label{co:eq}
For any $\varepsilon\in (0,1)$, there exists a metric $g_s$ defined as in Lemma \ref{lem:qis} such that
\begin{equation}\label{eq:traceeq}
(1-\varepsilon)\Tr(e^{-t\triangle_{g}})\leq \Tr(e^{-t\triangle_{g_s}})\leq (1+\varepsilon)\Tr(e^{-t\triangle_{g}}) .
\end{equation}
\end{corollary}
\begin{proof}
By \eqref{eq:metriceq}, we have for any $u\in C_c^\infty(X)$
\begin{equation}\label{eq:l2normseq}
(1-\varepsilon) \int| u|_g^2\,dv_g \leq \int | u|_{g_s}^2\,dv_{g_s} \leq (1+\varepsilon) \int|u|_g^2\,dv_g . 
\end{equation}
Now, combining \eqref{eq:l2normseq} with \eqref{eq:nablaeq}, we find that the eigenvalues of $\triangle_{g}$ and $\triangle_{g_s}$ agree up to a factor of $(1+\varepsilon)/(1-\varepsilon)$. Since the function $[0,\infty)\ni t\mapsto e^{-t}$ is Lipschitz continuous with Lipschitz constant $1$, the same holds true for the heat traces. Up to adapting the $\varepsilon$, this yields \eqref{eq:traceeq}.
\end{proof}
Let us now conclude the proof of Theorem \ref{theo:A} in the general case.
\begin{proof}[Proof of Theorem \ref{theo:A}.]
By Lemma \ref{lem:qis}, choose $g_s$ with arbitrarily small $\varepsilon>0$. The results of Section \ref{sec:cclAsep} give the leading asymptotics of $\Tr(e^{-t\triangle_{g_s}})$ in the three regimes; Corollary \ref{co:eq} then implies the same asymptotics for $\Tr(e^{-t\triangle_{g}})$, with the same constant (the factor $1+\mathrm{O}(\varepsilon)$ disappears as $\varepsilon\to 0^+$). The Weyl law follows by the Tauberian step already implemented in Section \ref{sec:cclAsep}.
\end{proof}

\subsection{Proof of Theorem~\ref{theo:B}}\label{sec_proof_theo:B}
In order to prove Theorem~\ref{theo:B}, the main idea is to decompose the manifold $X$ in an interior part, as well as two domains near the boundary that are of the form $D=[0,a]\times D_1$ with $D_1\subset M$ and $a>0$. Then one can deduce Theorem~\ref{theo:B} from the leading-order asymptotics of the heat kernel on all those three pieces.
This asymptotics can be determined using a sandwiching and monotonicity argument, comparing with the case with Dirichlet boundary conditions on the interfaces of those domains.

\begin{lemma}\label{le:DE}
Fix $a\in(0,1)$ and a piecewise smooth subset $D_1\subset M$. Set
$D=[0,a]\times D_1$ and $E=X\setminus([0,a]\times M)$. For a domain $K$, denote the Dirichlet heat kernel by $Z_K(t)=\Tr(e^{-t\triangle_K^{\mathrm D}})$ and the integral of the heat kernel on $X$ over $K$ by $Z'_K(t)=\int_K e_X(t,m,m)\,dv_g(m)$. 
\begin{itemize}
\item[(i)] Then for any $\beta>0$, we have as $t\to0$, 
\begin{equation}\label{eq:interior}
Z'_E(t)=(4\pi t)^{-\frac{n+1}{2}}\!\int_{E}dv_g+ \mathrm{o}\big(t^{-\frac{n+1}{2}}\big).
\end{equation}
\item[(ii)] If $\beta>\beta_c$, then
\begin{equation}\label{eq:DirichletcritD}
Z_D(t)\sim C\,\mathrm{Vol}_G(D_1)\,t^{-d_H/2}
\end{equation}
as $t\to 0^+$, where $C$ is the constant from the Weyl law. In the case $\beta=\beta_c$, \eqref{eq:DirichletcritD} holds with $t^{-d_H/2}$ replaced by $t^{-(n+1)/2}|\ln t|$. Furthermore, for any $\beta\geq\beta_c$, the same asymptotics hold when $Z_K$ is replaced by $Z'_K$.
\end{itemize}
\end{lemma}
\begin{proof}
(i) follows from the classical local heat kernel expansion. 

For (ii), let $\beta\geq\beta_c$ and note that the conical analysis of Section \ref{sec:heatcone} applied to $D$ with Dirichlet boundary conditions yields the Dirichlet asymptotics in \eqref{eq:DirichletcritD}. Here, the gluing between the conical part of $D$ and its interior part is performed as in Proposition~\ref{prop_parametrix}, whose proof applies verbatim with $X$ replaced by $D$: Lemma~\ref{lem_offdiag} holds for the Dirichlet realization of the Laplacian on the interior of $D$ (and of $E$, $F$), with the same test functions $\psi$, depending only on the variable $x$. Similarly, we also obtain this result for $D$ replaced by $X\setminus (D\cup E)=:F$. Note that we also have the asymptotics as in \eqref{eq:interior} for the Dirichlet heat kernel $Z_E(t)$ on $E$. Summing those asymptotics gives 
\begin{equation}
Z_D(t)+Z_E(t)+Z_F(t)\sim C\,\mathrm{Vol}(M)\,t^{-d_H/2}
\end{equation}
with $t^{-d_H/2}$ replaced by $t^{-(n+1)/2}|\ln t|$ when $\beta=\beta_c$. Here the constant $C$ is the constant from the Weyl asymptotics in Theorem~\ref{theo:A}. Domain monotonicity for Dirichlet heat kernels (see \cite[Thm.~2.1.6]{Dav}) gives the pointwise bounds
$0\leq e_D\leq e_X$ on $D$, and similarly for $E,F$. Hence $Z_K(t)\leq Z'_K(t)$ for $K=D,E,F$, while $Z'_D+Z'_{E}+Z'_{F}=\Tr(e^{-t\triangle_g})$.
A sandwiching argument transfers the leading-order terms from $Z_K$ to $Z'_K$ for $K\in\{D,E\}$.
\end{proof}
\begin{proof}[Proof of Theorem~\ref{theo:B}.]
In the case $\beta<\beta_c$, Lemma \ref{le:DE}(i) combined with Theorem~\ref{theo:A} shows that the Weyl measure has no singular part supported on the boundary of $X$. To this end, note that for any $\varepsilon>0$, we can find $a>0$ sufficiently small such that $\int_{E}dv_g\geq (1-\varepsilon)\int_{X}dv_g$. We also use that by Theorem~\ref{theo:A}, 
\begin{equation}
Z'_X(t)\sim(4\pi t)^{-\frac{n+1}{2}}\!\int_{X}dv_g 
\end{equation}
and that $Z'_{X\setminus E}(t)\ge 0$. 

Furthermore, for any continuous and compactly supported function $f$ on $X$, we have using the local heat kernel asymptotics in the interior of $X$,
$$
\Tr(e^{-t\triangle_g}f)=(4\pi t)^{-\frac{n+1}{2}} \int_X f\,dv_g + \mathrm{o} \big(t^{-\frac{n+1}{2}}\big),
$$
as $t\to 0^+$, hence the Weyl measure equals $dv_g/v_g(X)$.

\medskip

If $\beta\ge\beta_c$, it suffices to test against indicators $f=\mathds{1}_{[0,a]\times D_1}$. Lemma \ref{le:DE}(i)-(ii) yields
$$
\frac{\int_{[0,a]\times D_1} e_X(t,m,m)\,dv_g(m)}{\Tr(e^{-t\triangle_g})}
\xrightarrow[t\to 0^+]{} \frac{\mathrm{Vol}_G(D_1)}{\mathrm{Vol}_G(M)},
$$
hence the Weyl measure is $\delta_{x=0}\otimes dv_G/v_G(M)$.
\end{proof}

\section{Discussion and open problems}\label{sec_open}

In this article, motivated by the propagation of acoustic waves in gas giant planets, we derived the Weyl law for the Laplace-Beltrami operator on a smooth compact Riemannian $(n+1)$-dimensional manifold $X$ with boundary whose metric blows up near the boundary. Many new questions emerge. We present some of them. 

\paragraph{Quantum Ergodicity and Quantum Limits.} 
We have seen in Corollary~\ref{cor_qe} that, if $\beta \geq \beta_c$, then a density-one subsequence of eigenfunctions concentrates on $\partial X$. This is a preliminary result towards Quantum Ergodicity (QE).

Recall that, on a locally compact space $U$ endowed with a probability Radon measure $\mu$, given a self-adjoint nonnegative operator $T$ on $L^2(U,\mu)$, of discrete spectrum $\lambda_1 \leq \lambda_2 \leq \cdots \leq \lambda_j\leq\cdots\rightarrow+\infty$ associated with an orthonormal eigenbasis $\Phi=(\phi_j)_{j\in\N^*}$ of $L^2(U,\mu)$, a Quantum Limit (QL) of $\Phi$ is a probability Radon measure $\nu$ on $U$ that is a weak limit of a subsequence of the probability measures $\vert\phi_j\vert^2\,\mu$, i.e., there exists a subsequence $(j_k)_{k\in\N^*}$ such that 
\begin{equation}\label{defql}
\int_U f \vert\phi_{j_k}\vert^2 \, d\mu \underset{k \rightarrow+ \infty} \longrightarrow \int_U f \, d\nu\qquad \forall f \in C_c^0(U).
\end{equation}
We say that QE holds for $(T,\Phi)$ if there exists a QL $\nu$ on $U$ and a subsequence $(j_k)_{k\in\N^*}$ of density one such that \eqref{defql} holds.

One may wonder whether, when $\beta \geq \beta_c$, a QE property on $M$ would imply a QE property on $X$. Proving this fact certainly requires fine spectral properties of Schr\"odinger operators (see \cite{Anantharaman}). Besides, inspired by \cite[Theorem B]{CHT1}, we wonder what can be said on QLs supported on $\partial X=\{0\}\times M$: are they invariant under the geodesic flow of $(M,G)$ (where $G=g_1(0)$)? Defining QLs on $T^\star M$ will already be a challenge.

\paragraph{Inverse problems on spectra.}
A natural question is: does the spectrum of $X$ determine the spectrum of $M$? Attacking this problem certainly requires developing appropriate trace formulas, as in \cite{CdV-2}. 

\paragraph{Closed geodesics.}
Recalling that $G=g_1(0)$ where $g_1$ is defined by \eqref{gg1}, it is natural to view geodesics on $(M,G)$ as limits, in an appropriate sense, of geodesics on $(X,g)$. A natural question is then: do there exist some closed geodesics of $X$ accumulating on (converging to) closed geodesics of $\partial X = M$? We refer to \cite{CdV-1} for a similar question investigated in the framework of contact sub-Riemannian 3D manifolds. Here again, having appropriate trace formulas might be useful.

\paragraph{Observability properties.}
The study of the Weyl asymptotics is a first step towards solving some inverse problems. As explained in Section~\ref{sec1}, the knowledge of spectrum properties can already be used to check the validity of some models, but the main objective in the physical context would be the ability to reconstruct some features of the internal structure of the planets, based on the observation of acoustic waves. The feasibility of such an inverse problem is mathematically modeled by an observability inequality, which can be settled as follows for half-waves. Given any $T>0$ and any subset $\omega$ of $X$, we say that the observability property holds true for $(\omega,T)$ if there exists a positive constant $C_T(\omega)$ such that
\begin{equation}\label{obs}
\int_0^T \Big\Vert \mathds{1}_\omega \, e^{it\sqrt{\triangle_g}}\phi\Big\Vert_{L^2(X,dv_g)}^2\, dt \geq C_T(\omega)\Vert\phi\Vert_{L^2(X,dv_g)}^2 \qquad \forall\phi\in L^2(X,dv_g) .
\end{equation}

\noindent 
-- When $\beta<\beta_c$, we expect that \eqref{obs} holds as soon as $\omega$ is open and $(\omega,T)$ satisfies the Geometric Control Condition (GCC, see \cite{BardosLebeauRauch}), like in the classical case of a non-singular Riemannian metric. 

\smallskip

\noindent
-- When $\beta\geq\beta_c$, an obvious necessary condition for \eqref{obs} to hold is that $\omega$ contain an open neighborhood of a subset of $\partial X$. Indeed, take $\phi$ in \eqref{obs} to be a highfrequency eigenfunction and apply Corollary \ref{cor_qe}. We think that this condition is sufficient if moreover $(\omega,T)$ satisfies GCC. 

We note that, when $X$ is a closed ball in $\R^{n+1}$ (an idealized situation for an exactly round planet), GCC is never satisfied unless $\omega$ contains an open neighborhood of the \emph{whole} boundary of $X$, which is certainly not relevant for applications from the physical point of view. In this case where $X$ is a round ball, it is more interesting to take a small observation subset $\omega$, containing a small open subset of $\partial X$. But, as soon as $\omega$ is a proper subset of a half-ball, GCC (and thus \eqref{obs}) obviously fails due to trapped rays, propagating along a diameter never meeting $\omega$. In this deteriorated context, we wonder, however, whether \eqref{obs} is anyway satisfied if we restrict the inequality to radial waves or to surface waves, which are the most physically meaningful waves to be observed. 

\paragraph{Metrics that are singular on larger codimension submanifolds.}
In this paper, we have considered a class of singular metrics blowing up at the boundary of $X$, where the boundary can be seen as a codimension-one submanifold of $X$. 

In more general, let $X$ be a smooth compact manifold and let $Z$ be a submanifold of $X$ of codimension $m\in\N^*$, and consider the class of singular metrics $g$ on $X$ that are smooth on $X\setminus Z$ and that, near $Z$, are written as
$$ 
g = h +g_Z(x)r^{-\beta} 
$$
in a neighborhood of $Z$ assumed to be diffeomorphic to $Z\times B_m$ where $B_m$ is the unit ball of $\R^m_x$ (this holds if the normal bundle of $Z$ is trivial) equipped with the Euclidean metric $h$ and the polar coordinates $(r,\sigma)$, and $g_Z(x)$ is a metric on $Z$, parametrized by $x\in B_m$ and depending smoothly on $x$.
The techniques developed in our paper can certainly be extended to compute the Weyl asymptotics in such cases. 

\paragraph{Two-term expansion.}
In Theorem \ref{theo:A} we have computed an equivalent of the Weyl counting function $N(\lambda)$ as $\lambda\rightarrow+\infty$. As alluded in Remark \ref{rem_nextterm}, an interesting open issue is to compute the second term of the expansion. When $\beta<\beta_c$, the main term is the same as in the nonsingular case, and we wonder whether the second term would reflect the singularity and would be larger than the usual second term obtained in the classical smooth case (as, for instance, in the recent article \cite{FranckLarson}).
When $\beta\geq\beta_c$, we wonder whether the second term would reflect only what happens in the interior, or would still be affected by the singularity.

\appendix

\section{Appendix}

\subsection{Quasi-isometries} \label{app:quasi}

Let $X$ be a smooth manifold of dimension $n+1$, with boundary. 
Two metrics $g_1$ and $g_2$, smooth on $X\setminus \partial X$, are said to be $\varepsilon$-quasi-isometric if
$$
\left\vert \frac{g_1}{g_2} -1\right\vert \leq \varepsilon 
$$
uniformly on $X\setminus \partial X$.
For $i\in\{1,2\}$, let $\triangle_{g_i}$ be the Friedrichs extension of the Laplace-Beltrami operator on $(X,g_i)$ with core $C_0^\infty(X\setminus \partial X)$.

If $\triangle_{g_1}$ has a discrete spectrum $(\lambda_j^1)_{j\in\N^*}$ then $\triangle_{g_2}$ has also a discrete spectrum $(\lambda_j^2)_{j\in\N^*}$ and, for $\varepsilon \leq \frac{1}{2} $ (this condition is to get bounds on the inverse of $g_i$), there exists $C(n)>0$ such that, for every $j\in\N^*$,
$$ 
\left\vert \frac{\lambda_j^1}{\lambda_j^2} -1 \right\vert \leq C(n) \varepsilon . $$
Indeed, this estimate follows from the minimax characterization of the eigenvalues and from the comparison of the Rayleigh quotients, i.e., of the volumes and co-metrics.

\subsection{Karamata tauberian theorem and converse}\label{app:ka}

We recall the Karamata tauberian theorem (see \cite[Chapter XIII, Theorem 2]{Fel}).

\begin{theorem}\label{thm_karamata}
Let $\mu$ be a positive Radon measure on $\R^+$. If there exists $\alpha>0$ such that
$$ 
\int_0^{+\infty} e^{-t\lambda }d\mu (\lambda) \sim At^{-\alpha  }\quad (\textrm{resp., $A\vert\ln t \vert t^{-\alpha}$})
$$
as $t\rightarrow 0^+$, then
$$ 
\mu ([0,\lambda ])\sim \frac{A}{\Gamma (\alpha +1)} \lambda ^\alpha  \quad \left(\textrm{resp., $\frac{A}{\Gamma (\alpha +1)} \lambda ^\alpha  \ln \lambda$} \right)
$$
as $\lambda\rightarrow+\infty$.
\end{theorem}

We need a converse of Theorem \ref{thm_karamata}.
Let $f:\R^+ \rightarrow \R^+ $ be a nonincreasing function of class $C^1$, such that $f$ and $f'$ have a fast decay at infinity.
Let $(\lambda_j)_{j\in \N^*}$ be a nondecreasing sequence of positive real numbers. We define the counting function $N(\lambda)=\# \{j\in \N^*\ \mid\  \lambda_j \leq \lambda \}$, for any $\lambda\in\R$. 
The objective is to estimate the sum
$$ 
S=\sum_{j=1}^{+\infty} f(\lambda_j ) .
$$

\begin{proposition}\label{prop:kara-inv}
Assume that there exist $C>0$ and $\alpha>0$ such that $N(\lambda) \sim C \lambda^\alpha$ as $\lambda\rightarrow+\infty$.
For any $\varepsilon >0$, there exists $K(\varepsilon)>0$, depending on the counting function $N$ but not on $f$, such that
$$ 
\left\vert S -C\alpha \int_{\lambda_1}^{+\infty} f(\lambda)\lambda^{\alpha-1} \, d\lambda \right\vert 
\leq K(\varepsilon) f(\lambda_1) +\varepsilon \int_{\lambda_1}^{+\infty} f(\lambda) \lambda^{\alpha-1} \, d\lambda .
$$ 
\end{proposition}

\begin{proof}
Given any $\varepsilon >0$, let $\Lambda_0>0$ such that, for every $\lambda \geq \Lambda_0$,
\begin{equation}\label{Nlambdaeps}
(1-\varepsilon)C\lambda^\alpha \leq N(\lambda) \leq (1+\varepsilon) C \lambda^\alpha .
\end{equation}
Noting that $dN(\lambda) = \sum_{j=1}^{+\infty} \delta_{\lambda_j}$, using the Stieltjes integral, we have
$$
S = \sum _{\lambda_j <\Lambda_0}f(\lambda_j) + \int _{\Lambda_0}^{+\infty} f(\lambda)\, dN(\lambda) .
$$
Now, since $\sum _{\lambda_j <\Lambda_0}f(\lambda_j)\leq N(\Lambda_0) f(\lambda_1)$, we get by integration by parts, using the fast decay of $f$ at infinity, that
$$
S \leq N(\Lambda_0) \left( f(\lambda_1) - f(\Lambda_0) \right) - \int_{\Lambda_0}^{+\infty} f'(\lambda) N(\lambda) d\lambda .
$$

We derive an upper bound for $S$. A lower bound is obtained similarly.
Using \eqref{Nlambdaeps}, integrating by parts and using that $f(\Lambda_0)\leq f(\lambda_1)$, we obtain
\begin{equation*}
\begin{split}
- \int_{\Lambda_0}^{+\infty} f'(\lambda) N(\lambda) \, d\lambda 
& \leq -(1+\varepsilon) C \int_{\Lambda_0}^{+\infty} f'(\lambda)\lambda^\alpha \, d\lambda \\
& \leq (1+\varepsilon) C \left( f(\lambda_1)\Lambda_0^\alpha+  \alpha\int_{\Lambda_0}^{+\infty} f(\lambda)\lambda^{\alpha-1} \, d\lambda \right) .
\end{split}
\end{equation*}
Therefore,
$$ 
S\leq  f(\lambda_1) \left( N(\Lambda_0)+C(1+\varepsilon)\Lambda_0^\alpha \right) + (1+\varepsilon) C\alpha \int_{\lambda_1}^{+\infty} f(\lambda)\lambda^{\alpha-1} \, d\lambda
$$
and the result follows with $K(\varepsilon) =N(\Lambda_0) + C(1+\varepsilon) \Lambda_0^\alpha$.
\end{proof}

\subsection{Weyl limit point/limit circle criterion and the inverse-square model}\label{app:weyl}
Let $P=-\partial_x^2+q(x)$ on $L^2((0,\infty),dx)$ with domain $C_0^\infty(0,\infty)$, where $q\in C^\infty((0,\infty))$ is real-valued. Denote by $P_{\min}$ the closure of $P$ and by $P_{\max}$ the maximal operator with domain $\{u\in L^2\ \mid\ u,u'\ \text{abs.\ cont.},\ Pu\in L^2\}$.
By the Weyl limit point/limit circle classification (see \cite[Thms.~X.7, X.10, X.11]{RSII}), each endpoint $0$ and $+\infty$ is either \emph{limit point} (l.p.) or \emph{limit circle} (l.c.)%
\footnote{Fix an endpoint $a\in\{0,\infty\}$ and choose any $\lambda\in\mathbb C\setminus\mathbb R$ (e.g.\ $\lambda=i$). 
We say that $a$ is \emph{limit circle} (l.c.) if \emph{every} solution $u$ of $(P-\lambda)u=0$ lies in $L^2$ in a neighborhood of $a$
(i.e.\ on $(0,\varepsilon)$ when $a=0$, or on $(R,\infty)$ when $a=\infty$). 
Otherwise we say that $a$ is \emph{limit point} (l.p.). 
This dichotomy is independent of the choice of $\lambda\in\mathbb C\setminus\mathbb R$; see \cite[Thm.~X.10]{RSII}.
Equivalently (and useful when $\lambda\in\mathbb R$), $a$ is l.p.\ iff the space of $L^2$ solutions of $Pu=\lambda u$ near $a$ has dimension at most $1$, and $a$ is l.c.\ iff that space has dimension $2$.

If both endpoints are l.p., then $P_{\min}$ is essentially self-adjoint on $C_0^\infty(0,\infty)$. 
Each l.c.\ endpoint contributes $1$ to each deficiency index; hence if exactly one endpoint is l.c.\ (and the other is l.p.), self-adjoint extensions are obtained by imposing one real boundary condition at that endpoint (and similarly two conditions if both endpoints are l.c.). See \cite[Thms.~X.7, X.11]{RSII}.}
for the equation $Pu=0$, and:
\begin{itemize}
\item $P_{\min}$ is essentially self-adjoint iff $Pu=0$ has at most one $L^2$ solution at each endpoint, i.e., the equation is l.p. at both $0$ and $+\infty$;
\item if an endpoint is l.c., then the deficiency index contributed is $1$ and self-adjoint extensions are obtained by one real boundary condition at that endpoint (the other endpoint being l.p.).
\end{itemize}

\paragraph{The model potential.}
We consider $P_\omega = -\partial_x^2+\frac{C}{x^2}+\omega\,x^\beta$ with $C\geq 0$, $\beta>0$ and $\omega\geq 0$.
Near $x=0$ the potential is dominated by $C/x^2$; at $+\infty$ it is dominated by $\omega x^\beta$ when $\omega>0$ and by the free part when $\omega=0$.

\begin{lemma}
For the operator $P_\omega$:
\begin{enumerate}
\item $+\infty$ is always limit point. Indeed, if $\omega>0$ then $q(x)\to+\infty$ and there is at most one $L^2$ solution at $+\infty$; if $\omega=0$ then $P_0=-\partial_x^2+\frac{C}{x^2}$ is l.p.\ at $+\infty$ (as is the free Laplacian).
\item At $0$, the equation $P_\omega u=0$ is l.p.\ iff $C\geq\frac{3}{4}$ and l.c.\ iff $0\leq C<\frac{3}{4}$. Equivalently, write the Frobenius exponents $\gamma_\pm=\frac{1}{2}\pm\nu$, $\nu=\sqrt{C+\frac14}$.
Then $u(x)\sim x^{\gamma_\pm}$ as $x\to 0^+$, and $x^{\gamma_-}\in L^2(0,1)$ iff $\gamma_->-\frac{1}{2}$ (i.e.\ $C<\frac{3}{4}$). Thus there is a non-$L^2$ solution at $0$ iff $C\geq\frac{3}{4}$.
\end{enumerate}
Consequently, $P_\omega$ is essentially self-adjoint on $C_0^\infty(0,\infty)$ for every $\omega\geq 0$ iff $C\geq\frac{3}{4}$, and it has deficiency indices $(1,1)$ when $0\leq C<\frac{3}{4}$.
\end{lemma}

\begin{proof}
(1) For $\omega>0$, $q(x)\to+\infty$ implies l.p.\ at $+\infty$ (see \cite[Thm.~X.10]{RSII}). For $\omega=0$, the two linearly independent solutions of $u''=\frac{C}{x^2}u$ behave like $1$ and $x$ at $\infty$, hence are not in $L^2(1,\infty)$, so $+\infty$ is l.p. (2) Near $0$, the indicial equation $-\gamma(\gamma-1)+C=0$ yields the stated exponents. Since $\int_0^1 x^{2\gamma}\,dx<\infty$ iff $\gamma>-1/2$, the l.p./l.c.\ dichotomy at $0$ follows (cf.\ \cite[Thm.~X.10]{RSII}). The final statement is Weyl’s criterion: l.p.\ at both endpoints $\iff$ essential self-adjointness.
\end{proof}

\paragraph{Friedrichs extension for $0\leq C<\frac{3}{4}$.}
In the l.c.\ case at $0$, all self-adjoint extensions of $P_\omega$ are determined by one boundary condition at $x=0$. The \emph{Friedrichs extension} $P_\omega^{\mathrm F}$ is the one associated with the closure of the positive quadratic form
$$
\mathfrak q_\omega[u]=\int_0^{+\infty} \bigl(|u'|^2+\frac{C}{x^2}|u|^2+\omega x^\beta |u|^2\bigr)\,dx,\qquad u\in C_0^\infty(0,\infty),
$$
and it corresponds to killing the more singular asymptotic component at $0$: if $u(x)=a\,x^{\gamma_-}+b\,x^{\gamma_+}+o(x^{\gamma_+})$ as $x\to 0^+$, then $a=0$. Equivalently,
$$
\lim_{x\to 0^+}x^{-\gamma_-}u(x)=0\qquad\text{(boundary condition for $P_\omega^{\mathrm F}$)}.
$$
With this choice, the operator is positive and its domain is contained in $H^1_{\mathrm{loc}}(0,\infty)$ with finite energy $\mathfrak q_\omega[u]$ (see \cite[Thm.~X.23]{RSII}).\footnote{The Hardy inequality $\int_0^{+\infty} \frac{|u|^2}{x^2}\,dx\leq 4\int_0^{+\infty} |u'|^2\,dx$ shows that $\mathfrak q_\omega$ is closed for $C\geq 0$.}

\begin{remark}[Independence of $\omega$ for self-adjointness]
The classification at $0$ depends only on $C$; the term $\omega x^\beta$ is negligible there. At $+\infty$, both $\omega>0$ and $\omega=0$ give l.p., as discussed above. Thus the essential self-adjointness threshold is exactly $C\geq \frac{3}{4}$, independently of $\omega\geq 0$.
\end{remark}

%
%

\subsection{Local nature of the small-time asymptotics of heat kernels}\label{app:loc}

Let $(U,g)$ be a smooth Riemannian manifold and let $\triangle$ be the Laplace-Beltrami operator. 
For our needs (see Sections \ref{sec:cclAsep} and \ref{sec_proof_theo:B} and Appendix \ref{app_hormander}), $U=X\setminus\partial X$ with the metric $g$ .

Let $e_1$ and $e_2$ be two solutions of $ (\partial_t+ \triangle_x )e_i(t,x,y)=0 $ for $t>0$, satisfying $e_i(t,x,y)=e_i(t,y,x)$ for all $t>0$ and $(x,y)\in U\times U$ and
$$ 
\lim_{t\rightarrow 0^+} \int _U e_i(t,x,y )f(y) \, dv_g(y) = f(x) \qquad \forall x\in U\qquad \forall f\in C_0^\infty (U), 
$$  
for $i\in\{1,2\}$.

\begin{lemma}\label{lem:localheat}
We have $e_1 (t,\cdot,\cdot) - e_2(t,\cdot,\cdot)= \mathrm{O}(t^\infty)$ as $t\rightarrow 0^+$ in $C^\infty$ topology on $U\times U$. Moreover, denoting by $D$ the diagonal of $U\times U$, for $i\in\{1,2\}$, we have $e_i(t,\cdot,\cdot) =\mathrm{O}(t^\infty)$ as $t\rightarrow 0^+$ in $C^\infty$ topology on $U\times U\setminus D$.
\end{lemma}

This result reflects Kac's principle of ``not feeling the boundary", showing that the small-time asymptotic behavior of heat kernels is purely local. 
A detailed proof can be found in \cite[Section 3.2.1]{CHT2}. The idea comes from the paper \cite{Je-Sa}.
The proof uses the fact that the H\"ormander operator $P= 2\partial_t +\triangle_x +\triangle_y $
is hypoelliptic. Extending the kernels $e_i$ by $0$ for $t< 0$, we have $Pe_i=0$ on $\R \times U\times U\setminus D$ and $P(e_1-e_2)=0$ on $\R \times U\times U$, in the distributional sense. The result then follows by hypoellipticity. 

\subsection{$\triangle_g$ as a nonsmooth H\"ormander operator}\label{app_hormander}

Based on the mathematical model provided in Section \ref{sec_math_model}, near any point of the boundary of $X$ we have $X\simeq[0,1)\times\R^n$ with a local system of coordinates $(x,y)$, with $x\in[0,1)$ and $y=(y_1,\ldots,y_n)\in\R^n$, and we can write (locally)
\begin{equation}\label{horm}
\triangle_g = -\sum_{i=0}^n X_i^*X_i + V
\end{equation}
where $V(x,y)=\frac{C(x,y)}{x^2}$ is a potential and the $X_i$'s are vector fields given by
$$
X_0 = a_0(x,y)\,\partial_x,\qquad X_i = x^{\beta/2} a_i(x,y)\,\partial_{y_i}, \quad i\in\{1,\ldots,n\} .
$$
The functions $C$ and $a_i$, $i\in\{0,\ldots,n\}$ are smooth on $\R\times\R^n$ and $C(0,\cdot)=C_\beta$ and $a_i(0,\cdot)=1$ (they can be expressed in terms of the coefficients of the smooth Riemannian metric $g_1(x)$ on $M$ defined by \eqref{gg1}). 
The separable case corresponds to $a_0=1$ and $a_i$ not depending on $x$.

Expressed as \eqref{horm}, the operator $\triangle_g$ is then a H\"ormander operator, however nonsmooth unless $\beta\in 2\N^*$. Because of this lack of smoothness, many classical results cannot be applied here. 

When $\beta\in 2\N^*$, the above vector fields are smooth and define an almost-Riemannian geometry, in which the Weyl asymptotics of the almost-Riemannian Laplacian $\triangle_{aR} = -\sum_{i=0}^n X_i^*X_i$ (i.e., \eqref{horm} with $V=0$), of Grushin type, has been established in \cite{CHT3}.

With these preliminary remarks in mind, we then mention a few interesting facts hereafter.

\paragraph{Homogeneity.} In the above local coordinates, given any $\varepsilon>0$, we define the \emph{dilation}
$$
\delta_\varepsilon(x,y) = (\varepsilon x, \varepsilon^{1+\beta/2}y)\qquad \forall (x,y)\in[0,1)\times\R^n.
$$
In the separable case where $a_0=1$ and $a_i$ does not depend on $x$, for any $i\in\{1,\ldots,n\}$, we define $\widehat{X}_i = \lim_{\varepsilon\rightarrow 0} \varepsilon\delta_\varepsilon^*X_i = x^{\beta/2} a_i(0)\,\partial_{y_i}$, and we have
$$
\varepsilon\delta_\varepsilon^*X_0=X_0\qquad\textrm{and}\qquad
\varepsilon\delta_\varepsilon^*\widehat{X}_i=\widehat{X}_i\quad\forall i\in\{1,\ldots,n\}.
$$
In sR geometry, $\widehat{X}_i$ is the \emph{nilpotentization} of the vector field $X_i$ at the point identified with $(0,0)$. Extrapolating results of sub-Riemannian geometry that one can find in \cite{CHT2} to the case of $\beta>0$, denoting by $d_g$ the $g$-distance on $X$, one can show that $d_g((0,0),(x,y))$ divided by $\vert x\vert + \sum_{i=1}^n\vert y_i\vert^{1/(1+\beta/2)}$ is bounded above and below by some positive constants in a neighborhood of $(0,0)$. 
Noting that $1/(1+\beta/2) = 1-\alpha/2$, we thus recover \cite[Proposition 13]{MdH-1} and thus the result of Proposition \ref{lem_Hausdorff} and the fact that Hausdorff and Minkowski dimensions coincide.
In the non-separable case, we obtain the result by using quasi-isometries. 

\paragraph{Weyl law when $\beta \in 2\N^*$.} When $\beta\in 2\N^*$, we always have $\beta\geq\beta_c$, and $\beta=\beta_c$ if and only if $n=1$. Since the potential $1/x^2$ is homogeneous, combining results of \cite{CHT2,CHT3}, we recover the Weyl law established in Theorem \ref{theo:A}.

\paragraph{Weyl law when $\beta \notin 2\N^*$.} To establish the Weyl law in general sub-Riemannian cases, the approach developed in \cite{CHT3} consists of estimating singular integrals involving the heat kernel, by performing the so-called $(J+K)$-decomposition. Applying this approach to the nonsmooth operator in \eqref{horm} cannot be done directly because we miss a general hypoellipticity theory, valid for nonsmooth vector fields as above,
and a generalization of Lemma \ref{lem:localheat} (see Appendix \ref{app:loc}) to that context.

\paragraph{Acknowledgements.}
The authors thank Bernard Helffer for pointing out the reference \cite{Metivier_JEDP1975}, and Anton Ermakov for many insightful discussions on gas giants and ring seismology. 
We warmly thank Luc Hillairet for pointing out a gap in the proof of Proposition~\ref{prop_parametrix} in the published version of this article (see Remark~\ref{rem_gap}). 
C. Dietze expresses her deepest gratitude to Phan Th\`anh Nam and Laure Saint-Raymond for their continued support, and thanks Daniel Grieser for pointing out a mistake in the introduction. She acknowledges the support by the European Research Council via ERC CoG RAMBAS, Project No.~101044249 and the Engie Foundation. M.V. de Hoop carried out the work while he was an invited professor at the Centre Sciences des Donn\'{e}es at \'{E}cole Normale Sup\'{e}rieure, Paris. He acknowledges the support of the Simons Foundation under the MATH $+$ X Program and the sponsors of the GeoMathematical Imaging Group at Rice University. E. Tr\'elat acknowledges the support of the grant ANR-23-CE40-0010-02 (Einstein-PPF).

\paragraph{Conflicts of interests/Competing interests.}
We declare that we have no conflict of interest.
We declare that there is no data supporting the results reported in our paper.
Besides: data citations are accessible from Mathscinet.

%

\end{document}